\documentclass[12pt]{article}
\usepackage{latexsym}
\usepackage{amssymb}
\usepackage{amsmath}
\usepackage{rotating}
\usepackage{multirow}
\usepackage{mathrsfs}
\usepackage{wasysym}
\usepackage{esint}
\usepackage{rawfonts}
\input{prepictex}
\input{pictex}
\input{postpictex}
\usepackage[OT2,OT1]{fontenc}
\def\cyr{%
\renewcommand\rmdefault{wncyr}%
\renewcommand\sfdefault{wncyss}%
\renewcommand\encodingdefault{OT2}%
\normalfont
\selectfont}
\DeclareMathAlphabet{\zap}{OT1}{pzc}{m}{it}
\DeclareTextFontCommand{\textcyr}{\cyr}
\def\be{\begin{equation}}
\def\ee{\end{equation}}
\def\bea{\begin{eqnarray*}}
\def\eea{\end{eqnarray*}}

\def\CC{\mathbb C}

\newtheorem{main}{Theorem}

\DeclareMathOperator{\Iso}{Iso}

\newtheorem{lem}{Lemma}
\newtheorem{prop}{Proposition}

\newenvironment{proof}{\medskip \noindent
{\bf Proof.}}{\hfill \rule{.5em}{1em}
\\}

\def\ZZ{{\mathbb Z}}
\def\RR{{\mathbb R}}
\def\CP{{\mathbb C \mathbb P}}
\begin{document}

\title{On Einstein, Hermitian $4$-Manifolds}

\author{Claude LeBrun\thanks{Supported 
in part by  NSF grant DMS-0905159.}\\SUNY Stony
 Brook}

\date{}
\maketitle

 \begin{abstract}	
 Let $(M,h)$ be a compact $4$-dimensional Einstein manifold,
 and suppose that $h$ is Hermitian with respect to some
 complex structure $J$ on $M$.  Then  either $(M,J,h)$ is  K\"ahler-Einstein,
 or else, up to   rescaling  and isometry,  it is one of the following two exceptions:
 the Page metric on $ \CP_2 \# \overline{\CP}_2$, or the 
Einstein metric on $\CP_2 \# 2\overline{\CP}_2$ discovered 
in \cite{chenlebweb}. 
 \end{abstract}

\section{Introduction}

Recall \cite{bes} that a Riemannian metric is said to be {\em Einstein}
if it has constant Ricci curvature. A central problem of modern 
Riemannian geometry is to determine   which smooth 
compact manifolds admit Einstein  metrics, and  to precisely understand the moduli
space of these metrics when they do exist. 

The theory of K\"ahler-Einstein metrics provides the richest currently available  source for  
Einstein metrics  on  compact manifolds. This story 
becomes particularly compelling in real dimension $4$,   not only because of 
 the mature state of the 
theory of K\"ahler-Einstein metrics on compact complex surfaces \cite{aubin,s,tian,ty,yau}, 
but also because 
gauge-theoretic phenomena unique to this dimension
sometimes allow one to show   \cite{hit,lmo} that {\em every} Einstein metric on certain 
smooth compact $4$-manifolds must actually be K\"ahler-Einstein.  

In fact, there are, up to rescaling and isometries, 
 only  two known examples of  Einstein metrics on 
compact complex surfaces
which are  {\em not} K\"ahler. The older
and better-known  of these is the Page metric \cite{page} on 
$\CP_2 \# \overline{\CP}_2$. 
The second, of more recent provenance    \cite{chenlebweb},   is a  metric 
on $\CP_2 \# 2\overline{\CP}_2$ discovered by the present author in collaboration
with Xiuxiong Chen and Brian Weber; for a somewhat different proof, see
\cite{lebhem10}. 
Both the of these metrics have holonomy $SO(4)$, and so are   non-K\"ahler
in the most  definitive, intrinsic sense. 
However, both are nonetheless {\em conformally} K\"ahler, in the sense that 
each is  obtained from some K\"ahler metric by multiplying by a smooth positive 
function. In particular, both are  Hermitian metrics on compact complex surfaces. 
The purpose of this article is prove that, up to homothety,  no other other
such examples can exist: 

\begin{main} \label{jubilo}
Let $(M^4 , J )$ be a compact complex surface, and 
suppose that $h$ is an Einstein metric on $M$ which is 
{Hermitian} with respect to $J$, in the sense that $h(\cdot , \cdot )= h(J\cdot, J\cdot )$.
 Then either 
\begin{itemize}
\item $(M ,J, h)$  is K\"ahler-Einstein; or 
\item $M \approx \CP_2 \# \overline{\CP}_2$, and $h$ is a constant times the Page metric; or 
\item  $M \approx \CP_2 \# 2\overline{\CP}_2$ and $h$ is a constant times the  metric of \cite{chenlebweb}. 
\end{itemize}
\end{main}

In \cite{lebhem}, 
the present author previously  proved 
a tantalizing partial result in this direction:

\begin{prop} \label{aarhus} 
Let $(M^4 , J )$ be a compact complex surface, and 
suppose that $h$ is an Einstein metric on $M$ which is 
{Hermitian} with respect to $J$. 
 Then $h$ is conformal 
to a $J$-compatible K\"ahler metric $g$.  Moreover, 
if $h$ is not itself K\"ahler, then 
\begin{itemize}
\item $(M ,J)$ has
${c_1}> 0$;
\item $M \approx \CP_2 \# {k}\overline{\CP}_2$, $k = 1,2,3$; 
\item $h$ has positive Einstein constant; 
\item $g$ is an extremal K\"ahler metric;
\item$g$ has   scalar curvature $s > 0$; and
\item after  suitable normalization, $h=s^{-2}g$.
\end{itemize}
Moreover, the isometry group of $(M,h)$ contains a $2$-torus. 
\end{prop}
The proof  of this assertion consists of three main steps. First, 
the Riemannian  Goldberg-Sachs Theorem \cite{aggs,gs,nur,pb} implies
that the self-dual Weyl curvature $W_+$ has a repeated eigenvalue at
each point. If $h$ is not itself K\"ahler, a result of 
Derdzi{\'n}ski  \cite[Theorem 2]{derd} then implies that this Einstein metric
can be written as $h=s^{-2}g$ for an  extremal K\"ahler metric $g$
with non-constant positive scalar curvature $s$. 
Finally, if $\rho$ denotes the Ricci form of $g$, 
one can show \cite[Proposition 2]{lebhem}   that 
$\rho + 2i\partial\overline{\partial} \log s$ is  a positive $(1,1)$-form.
 Hence $c_1  > 0$, 
and $(M,J)$ is a Del Pezzo surface.
Since $M$ moreover admits an extremal K\"ahler metric
of non-constant scalar curvature,  its Lie algebra of holomorphic
vector fields must be both non-trivial and non-semi-simple.
The
classification of Del Pezzo surfaces \cite{delpezzo,cubic}
therefore implies that $M$ must be the blow-up of $\CP_2$ at  
one, two, or three points general position. 

For the one-point blow-up, the Page metric is the only possibility \cite{lebhem}. 
Indeed, 
for any extremal K\"ahler metric on a compact complex manifold, the 
the identity component  of the isometry group is necessarily \cite[Theorem 3]{calabix2}
 a maximal compact subgroup of the identity component of the complex automorphism group. 
If  $M=\CP_2 \# \overline{\CP}_2$,  and if $g$ is a conformally Einstein K\"ahler
metric on $M$, it then follows that 
 $\Iso_0(M,g) \cong U(2)$.
 But since  $\Iso_0(M,g)$   automatically preserves the scalar curvature $s$, 
it also acts  isometrically on the Einstein metric $h=s^{-2}g$. Thus $h$
is a cohomogeneity-one Einstein metric, and  the work of B\'erard Bergery 
\cite[Th\'eor\`eme 1.8]{beber3} then shows that 
it must actually be a constant times the  Page metric. 

Unfortunately, the remaining 
cases of $M=\CP_2 \# 2\overline{\CP}_2$ and  $\CP_2 \# 3\overline{\CP}_2$
are not amenable to  elementary arguments of this flavor.  Instead, this article 
will  prove the   following pair of results by  a variational method:

\begin{main} \label{uni2}
Modulo rescalings and biholomorphisms, there is exactly one conformally  
K\"ahler, Einstein metric
$h$ on $M=\CP_2\# 2 \overline{\CP}_2$. This metric  coincides with 
the metric of \cite{chenlebweb}, and is   characterized
by the fact that the conformally related K\"ahler metric $g$ 
minimizes the $L^2$-norm of the scalar curvature among
all K\"ahler metrics on $M$. 
\end{main}

\begin{main}  \label{uni3}
Modulo rescalings and biholomorphisms, there is only one conformally 
K\"ahler, Einstein metric
$h$ on $M=\CP_2\# 3 \overline{\CP}_2$. This metric is actually K\"ahler-Einstein, 
and is exactly the metric discovered by Siu \cite{s}. 
\end{main}

Our approach  stems from the  study of the 
squared $L^2$-norm 
$${\mathcal C}(g)= \int_M s^2_g ~d\mu_g$$
of the scalar curvature, restricted to the space of K\"ahler metrics. 
An even more restricted version of this problem was introduced by Calabi \cite{calabix}, 
who constrained  $g$ to only vary in a  fixed  K\"ahler class
$[\omega ] \in H^{2} (M,\RR)$.  Calabi called the  critical metrics of
his restricted problem  {\em extremal K\"ahler metrics}, and showed
that the relevant Euler-Lagrange equations 
 are equivalent to requiring that
$\nabla^{1,0}s$ be a holomorphic vector field. In fact, 
every extremal K\"ahler metric turns out to be an 
absolute minimizer for the Calabi problem, and the proof of this  \cite{xxel}
moreover implies that any K\"ahler metric $g$ with K\"ahler class $[\omega ]=\Omega$
satisfies  the sharp estimate 
\begin{equation}
\label{sharp}
{\mathcal C}(g)  \geq 32\pi^2 {\mathcal A} (\Omega ) ~,
\end{equation}
where 
$$
{\mathcal A} (\Omega ):=
\frac{(c_1\cdot \Omega )^2}{\Omega^2}+ \frac{1}{32\pi^2}\|{\mathfrak F}(\Omega) \|^2~,
$$
and where  equality occurs iff $g$ is an extremal K\"ahler metric. Here 
$${\mathfrak F}(\Omega ) : H^0(M, {\mathcal O}(T^{1,0}M))\to \CC$$
denotes the Futaki invariant, and the relevant norm is the one induced by the 
$L^2$-norm on the space of  holomorphy potentials  \cite{fuma0}. In particular, 
for any extremal K\"ahler metric $g$ with K\"ahler class $\Omega$, one has 
\begin{eqnarray*}
\int_M s_0^2 ~d\mu_g &=& 32\pi^2 \frac{(c_1\cdot \Omega )^2}{\Omega^2}\\
\int_M (s-s_0)^2_g ~d\mu_g &=& \|{\mathfrak F}(\Omega)\|^2
\end{eqnarray*}
where 
$$
s_0= \fint_M s ~d\mu_g = \frac{\int s~d\mu_g}{\int d\mu_g}
$$
denotes the average scalar curvature.

 Essentially because    quadratic curvature functionals are scale-invariant in 
 real dimension $4$,  one  has 
  ${\mathcal A}(\lambda \Omega ) = {\mathcal A}( \Omega )$ for every
 $\lambda\in \RR^+$. Thus, if  ${\zap K} \subset H^2 (M, \RR)$ is the K\"ahler cone of 
$(M,J)$, 
  then letting $\check{\zap K}$ denote 
  ${\zap K}/\RR^+$, where the positive real numbers $\RR^+$ act 
  by scalar multiplication,  we often choose to  consider ${\mathcal A}$ as
  a function  ${\mathcal A}:\check{\zap K}\to \RR$.
  But from either point of view, 
  the following variational 
  principle \cite{chenlebweb, lebhem}  underpins our entire approach: 
      
 \begin{prop} \label{critical} 
Suppose that  $h$ is an Einstein metric on $M$ which is conformally related 
to a $J$-compatible K\"ahler metric $g$ with K\"ahler class $[\omega ]\in {\zap K}$. 
Then $[\omega ]$ is a critical point of ${\mathcal A}$.
 \end{prop}

Fortunately, the formula for ${\mathcal A}$ can be found explicitly, although the 
actual expression is  complicated  enough that 
a program like {\em Mathematica}  is of enormous help in 
reliably obtaining the correct answer. 
In an earlier investigation, Maschler \cite{gideon2} used this  to  marshal  
overwhelming numerical evidence in favor of the conjecture
that $\mathcal A$ has exactly one critical point on either 
$\CP_2\#2\overline{\CP}_2$ and on
$\CP_2\#3\overline{\CP}_2$. In this paper, we will instead 
examine explicit formulas for  appropriate second derivatives of ${\mathcal A}$,
and observe that these 
imply that the function is strictly convex on certain line segments.
This allows one to show that a critical point must be invariant
under certain finite groups of automorphisms, and  leads to water-tight 
proofs of the uniqueness of the critical point. 
Theorems \ref{jubilo}, \ref{uni2} and
\ref{uni3}   then follow from
the uniqueness, up to isometry,  of extremal K\"ahler metrics
in a fixed K\"ahler class \cite{xxgang2,donaldsonk1,mabuniq}.

\section{The case of  $\CP\# 2 \overline{\CP}_2$} \label{comp2} 

In this section, we will show that 
$${\mathcal A}: \check{\zap K}\to \RR$$
 has only one critical point
  when  $M=\CP\# 2 \overline{\CP}_2$. We begin
by fixing a K\"ahler class, normalized by rescaling so that the proper
transform of the projective line between the two blow-up points has area $1$:

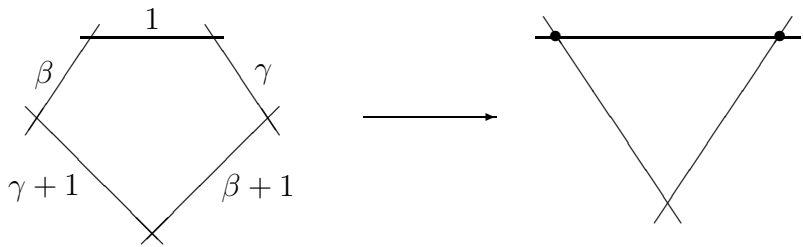
\begin{center}
\begin{picture}(240,80)(0,3)
\put(-7,70){\line(1,0){54}}
\put(40,75){\line(2,-3){28}}
\put(0,75){\line(-2,-3){28}}
\put(68,44){\line(-1,-1){52}}
\put(-28,44){\line(1,-1){52}}
\put(20,77){\makebox(0,0){$1$}}
\put(-21,56){\makebox(0,0){$\beta$}}
\put(62,56){\makebox(0,0){$\gamma$}}
\put(-21,13){\makebox(0,0){$\gamma+1$}}
\put(60,13){\makebox(0,0){$\beta+1$}}
\put(100,40){\vector(1,0){50}}
\put(210,0){\line(2,3){52}}
\put(220,0){\line(-2,3){52}}
\put(165,70){\line(1,0){100}}
\put(172.6,70.5){\circle*{4}}
\put(257.4,70.5){\circle*{4}}
\end{picture}
\end{center}

\bigskip 
\noindent 
The pair of positive real numbers  $(\beta, \gamma)$, representing the areas
of the two other exceptional divisors,  now provides us with global 
coordinates on the reduced K\"ahler cone $\check{\zap K}= {\zap K}/\RR^+$, 
thereby giving us a diffeomorphism  $\check{\zap K}\approx \RR^+\times \RR^+$.

Let us take the two blown-up points to be $[1,0,0], [0,1,0] \in \CP_2$, 
and fix the  maximal torus 
$$\left[ \begin{array}{ccc}
e^{i\theta}&&\\&e^{i\phi}&\\&&1
\end{array}\right]$$
in the automophism group. Then, for any $T^2$-invariant metric, 
the moment map of the torus action will take values in a pentagon,
which after translation becomes the following: 
 
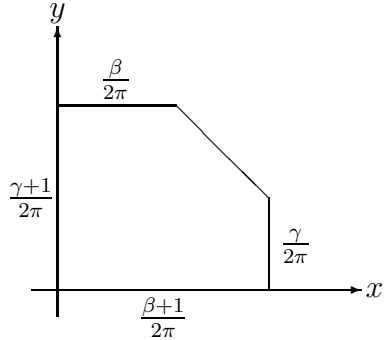
\begin{center}
\begin{picture}(120,120)(0,0)
\put(0,10){\vector(1,0){125}}
\put(10,0){\vector(0,1){110}}
\put(130,10){\makebox(0,0){$x$}}
\put(10,115){\makebox(0,0){$y$}}
\put(50,0){\makebox(0,0){$\frac{\beta +1}{2\pi}$}}
\put(0,45){\makebox(0,0){$\frac{\gamma +1}{2\pi}$}}
\put(100,27){\makebox(0,0){$\frac{\gamma }{2\pi}$}}
\put(32,90){\makebox(0,0){$\frac{\beta }{2\pi}$}}
\put(10,80){\line(1,0){45}}
\put(55,80){\line(1,-1){35}}
\put(90,10){\line(0,1){35}}
\end{picture}
\end{center}
Let ${\mathfrak F}_1$ and ${\mathfrak F}_2$ be Futaki invariants of
this K\"ahler class with respect to the vector fields with Hamiltonians 
$-x$ and $-y$. Then \cite{ls} for any $T^2$-invariant metric, 
\begin{eqnarray*}
{\mathfrak F}_1&=& \int_M x(s-s_0) d\mu \\
&=& \frac{1}{V}\left[ 
(\beta -2\gamma) (\frac{1}{3} + \gamma + \gamma^2) + \gamma (\gamma - \beta) (2+ \beta + 2\gamma ) 
\right]\\
{\mathfrak F}_2&=& \int_M y(s-s_0) d\mu \\
&=& \frac{1}{V}\left[ 
(\gamma -2\beta) (\frac{1}{3} + \beta + \beta^2) + \beta (\beta - \gamma) (2+ \gamma + 2\beta ) 
\right]\\
\end{eqnarray*}
where 
$$
V= \beta \gamma + \beta + \gamma +  \frac{1}{2}. 
$$

Note that, by Archimedes' principle, 
 the push-forward of the volume measure of $M$ is exactly $4\pi^2$
times the Euclidean measure  on the moment polygon. 
Thus, for example, the average values  $x_0$ and $y_0$ of the Hamiltonians  
 $x$ and $y$ on $M$ are also  the $x$ and $y$ coordinates of
 the barycenter of the moment pentagon. This same observation also 
 makes it straightforward to compute
 the following useful constants:
\begin{eqnarray*}
A&:=& \int_M (x-x_0)^2d\mu \\
&=& \frac{1 +  6 (1 + \beta)[  \beta + \beta^2 + \beta^3+
    \gamma  (1 + 4 \beta + 4 \beta^2 + 2 \beta^3)+
    \gamma^2 (1 + \beta)^3]}{288\pi^2 V} \\
   B&:=&
    \int_M (y-y_0)^2d\mu\\
   &=&  \frac{1 +  6 (1 + \gamma)[  \gamma + \gamma^2 + \gamma^3+
    \beta  (1 + 4 \gamma + 4 \gamma^2 + 2 \gamma^3)+
    \beta^2 (1 + \gamma)^3] }{288\pi^2 V}
  \\  C &:=&  \int_M (x-x_0)(y-y_0)d\mu\\
 &=& -\frac{1 + 6 (1 + \beta) (1+\gamma) (\beta+ \gamma  + 3 \beta   \gamma)}{576\pi^2 V}
\end{eqnarray*}
Then 
$$\| \mathfrak F\|^2=  \frac{B{\mathfrak F}_1^2 - 2C{\mathfrak F}_1{\mathfrak F}_2+A{\mathfrak F}_2^2 }{AB-C^2}~,
$$
and 
$$
{\mathcal A}(\Omega)  = \frac{(c_1\cdot \Omega )^2}{\Omega^2} +  \frac{1}{32\pi^2}\frac{B{\mathfrak F}_1^2 - 2C{\mathfrak F}_1{\mathfrak F}_+A{\mathfrak F}_2^2 }{AB-C^2}
$$
can now be shown to  explicitly be given by 

\bigskip 

\noindent
$\displaystyle 
3 \Big[3 + 28 \gamma + 96 \gamma^2 + 168 \gamma^3 + 164 \gamma^4 + 80 \gamma^5 + 16 \gamma^6 + 
     16 \beta^6 (1 + \gamma)^4 + 
     16 \beta^5 (5 + 24 \gamma + 43 \gamma^2 + 37 \gamma^3 + 15 \gamma^4 + 2 \gamma^5) + 
     4 \beta^4 (41 + 228 \gamma + 478 \gamma^2 + 496 \gamma^3 + 263 \gamma^4 + 60 \gamma^5 + 
        4 \gamma^6) + 
     8 \beta^3 (21 + 135 \gamma + 326 \gamma^2 + 392 \gamma^3 + 248 \gamma^4 + 74 \gamma^5 + 
        8 \gamma^6) + 
     4 \beta (7 + 58 \gamma + 176 \gamma^2 + 270 \gamma^3 + 228 \gamma^4 + 96 \gamma^5 + 16 \gamma^6) + 
     4 \beta^2 (24 + 176 \gamma + 479 \gamma^2 + 652 \gamma^3 + 478 \gamma^4 + 172 \gamma^5 + 
        24 \gamma^6)\Big]\Big/\\
        \Big[ 1 + 10 \gamma + 36 \gamma^2 + 64 \gamma^3 + 60 \gamma^4 + 24 \gamma^5 + 
   24 \beta^5 (1 + \gamma)^5 + 12 \beta^4 (1 + \gamma)^2 (5 + 20 \gamma + 23 \gamma^2 + 10 \gamma^3) + 
   16 \beta^3 (4 + 28 \gamma + 72 \gamma^2 + 90 \gamma^3 + 57 \gamma^4 + 15 \gamma^5) + 
   12 \beta^2 (3 + 24 \gamma + 69 \gamma^2 + 96 \gamma^3 + 68 \gamma^4 + 20 \gamma^5) + 
   2 \beta (5 + 45 \gamma + 144 \gamma^2 + 224 \gamma^3 + 180 \gamma^4 + 60 \gamma^5)\Big]
$

\begin{lem}\label{convex2}
The 
restriction of ${\mathcal A}$ to
a line segment $\beta + \gamma=\mbox{\bf const}$ is always   a strictly convex function
on the interval where  $\beta , \gamma > 0$.  
\end{lem}

\begin{proof}
Machine-assisted computation reveals that 
$
\left( \frac{\partial}{\partial \beta}-\frac{\partial}{\partial \gamma}\right)^2 {\mathcal A}$
 is given by 
 
\bigskip 

\noindent 
$\displaystyle
24 \Big[ 2304 \beta^{16} (1 + \gamma)^{12} + 
     6912 \beta^{15} (1 + \gamma)^9 (3 + 11 \gamma + 12 \gamma^2 + 5 \gamma^3) + 
     576 \beta^{14} (1 + \gamma)^6 (167 + 1166 \gamma + 3343 \gamma^2 + 5064 \gamma^3 + 
        4371 \gamma^4 + 2070 \gamma^5 + 431 \gamma^6) + 
     384 \beta^{13} (1 + \gamma)^3 (774 + 7782 \gamma + 35067 \gamma^2 + 92491 \gamma^3 + 
        157494 \gamma^4 +
         180246 \gamma^5 + 139303 \gamma^6 + 70461 \gamma^7 + 
        21276 \gamma^8 + 2940 \gamma^9) +
          (1 + 4 \gamma + 6 \gamma^2 + 4 \gamma^3)^2 (4 + 
        62 \gamma + 457 \gamma^2 + 1922 \gamma^3 + 4910 \gamma^4 +
          7992 \gamma^5 + 8520 \gamma^6 + 
        5976 \gamma^7 + 2808 \gamma^8 + 864 \gamma^9 + 144 \gamma^{10}) + 
      2 \beta (1 + 4 \gamma + 6 \gamma^2 + 4 \gamma^3)^2 (47 + 620 \gamma + 4023 \gamma^2 + 
        15275 \gamma^3 + 35790 \gamma^4 +
          54420 \gamma^5 + 55668 \gamma^6 + 38844 \gamma^7 + 
        18504 \gamma^8 + 5616 \gamma^9 + 864 \gamma^{10}) + 
     96 \beta^{12} (1 + \gamma)^2 (6943 + 76784 \gamma + 389911 \gamma^2 + 1189166 \gamma^3 + 
        2406813 \gamma^4 +
         3382480 \gamma^5 + 3355525 \gamma^6 + 2334978 \gamma^7 + 
        1104488 \gamma^8 + 332712 \gamma^9 + 
          58398 \gamma^{10} + 7572 \gamma^{11} + 
        1938 \gamma^{12} + 312 \gamma^{13} + 24 \gamma^{14}) + 
     96 \beta^{11} (11727 + 153845 \gamma + 945314 \gamma^2 + 3583138 \gamma^3 + 
        9309806 \gamma^4 + 17464406 \gamma^5 + 
        24300890 \gamma^6 + 25383726 \gamma^7 + 
        19929243 \gamma^8 + 11686165 \gamma^9 + 5084700 \gamma^{10} + 
        1671216 \gamma^{11} + 
        457080 \gamma^{12} + 120384 \gamma^{13} + 27936 \gamma^{14} + 4104 \gamma^{15} + 
        288 \gamma^{16}) + 
     16 \beta^{10} (91367 + 1213674 \gamma + 7595262 \gamma^2 + 29525748 \gamma^3 + 
        79271358 \gamma^4 + 
       154934784 \gamma^5 + 
         226825122 \gamma^6 + 
        252558972 \gamma^7 + 215533575 \gamma^8 +
         141830454 \gamma^9 + 
        73002384 \gamma^{10} + 
         30508200 \gamma^{11} + 10969860 \gamma^{12} + 
      3434616 \gamma^{13} + 837216 \gamma^{14} + 129168 \gamma^{15} + 9504 \gamma^{16}) + 
     32 \beta^9 (46177 + 626348 \gamma + 4023003 \gamma^2 + 16156443 \gamma^3 + 
        45149508 \gamma^4 + 92636856 \gamma^5 + 
        143842384 \gamma^6 + 172157314 \gamma^7 + 
        160870371 \gamma^8 + 118918740 \gamma^9 + 70915227 \gamma^{10} + 
        35058495 \gamma^{11} + 14629998 \gamma^{12} + 5009448 \gamma^{13} + 1277280 \gamma^{14} + 
        206064 \gamma^{15} + 15840 \gamma^{16}) + 
     16 \beta^8 (73136 + 1020846 \gamma + 6782175 \gamma^2 + 28350072 \gamma^3 + 
        83038113 \gamma^4 + 179966490 \gamma^5 + 297854048 \gamma^6 + 
        384202536 \gamma^7 + 392256648 \gamma^8 + 321740742 \gamma^9 + 
        215533575 \gamma^{10} + 119575458 \gamma^{11} + 54779814 \gamma^{12} + 
        19939536 \gamma^{13} + 5297652 \gamma^{14} + 890352 \gamma^{15} + 71280 \gamma^{16}) + 
     16 \beta^7 (45406 + 656739 \gamma + 4546158 \gamma^2 + 19916894 \gamma^3 + 
        61505100 \gamma^4 + 141399807 \gamma^5 + 249914844 \gamma^6 + 
        346874028 \gamma^7 + 384202536 \gamma^8 + 344314628 \gamma^9 + 
        252558972 \gamma^{10} + 152302356 \gamma^{11} + 74571048 \gamma^{12} + 
        28478448 \gamma^{13} + 7856640 \gamma^{14} + 1371168 \gamma^{15} + 114048 \gamma^{16}) + 
     \beta^2 (1065 + 20598 \gamma + 195144 \gamma^2 + 1180776 \gamma^3 + 5035692 \gamma^4 + 
        15925800 \gamma^5 + 38527448 \gamma^6 + 72738528 \gamma^7 + 108514800 \gamma^8 + 
        128736096 \gamma^9 + 121524192 \gamma^{10} + 90750144 \gamma^{11} + 
        52840512 \gamma^{12} + 23322240 \gamma^{13} + 7398144 \gamma^{14} + 1513728 \gamma^{15} + 
        152064 \gamma^{16}) + 
     8 \beta^5 (16386 + 259536 \gamma + 1990725 \gamma^2 + 9755446 \gamma^3 + 
        33937098 \gamma^4 + 88398864 \gamma^5 + 177968682 \gamma^6 + 282799614 \gamma^7 + 
        359932980 \gamma^8 + 370547424 \gamma^9 + 309869568 \gamma^{10} + 
        209572872 \gamma^{11} + 112623264 \gamma^{12} + 46332864 \gamma^{13} + 
        13651776 \gamma^{14} + 2550528 \gamma^{15} + 228096 \gamma^{16}) + 
     2 \beta^3 (3769 + 67451 \gamma + 590388 \gamma^2 + 3311320 \gamma^3 + 
        13161720 \gamma^4 + 39021784 \gamma^5 + 89001976 \gamma^6 + 159335152 \gamma^7 + 
        226800576 \gamma^8 + 258503088 \gamma^9 + 236205984 \gamma^{10} + 
        171990624 \gamma^{11} + 98197056 \gamma^{12} + 42587904 \gamma^{13} + 
        13234176 \gamma^{14} + 2630016 \gamma^{15} + 253440 \gamma^{16}) + 
     4 \beta^6 (87921 + 1326312 \gamma + 9631862 \gamma^2 + 44500988 \gamma^3 + 
        145571257 \gamma^4 + 355937364 \gamma^5 + 671754360 \gamma^6 + 
        999659376 \gamma^7 + 1191416192 \gamma^8 + 1150739072 \gamma^9 + 
        907300488 \gamma^{10} + 583221360 \gamma^{11} + 300655152 \gamma^{12} + 
        119521344 \gamma^{13} + 34128576 \gamma^{14} + 6168960 \gamma^{15} + 
        532224 \gamma^{16}) + 
      2 \beta^4 (18413 + 308550 \gamma + 2517846 \gamma^2 + 13161720 \gamma^3 + 
        48864336 \gamma^4 + 135748392 \gamma^5 + 291142514 \gamma^6 + 
        492040800 \gamma^7 + 664304904 \gamma^8 + 722392128 \gamma^9 + 
        634170864 \gamma^{10} + 446870688 \gamma^{11} + 248402688 \gamma^{12} + 
        105206400 \gamma^{13} + 31888800 \gamma^{14} + 6148224 \gamma^{15} + 
        570240 \gamma^{16})\Big]\Big/ \\
        \Big[1 + 10 \gamma + 36 \gamma^2 + 64 \gamma^3 + 60 \gamma^4 + 
   24 \gamma^5 + 24 \beta^5 (1 + \gamma)^5 + 
   12 \beta^4 (1 + \gamma)^2 (5 + 20 \gamma + 23 \gamma^2 + 10 \gamma^3) + 
   16 \beta^3 (4 + 28 \gamma + 72 \gamma^2 + 90 \gamma^3 + 57 \gamma^4 + 15 \gamma^5) + 
    12 \beta^2 (3 + 24 \gamma + 69 \gamma^2 + 96 \gamma^3 + 68 \gamma^4 + 20 \gamma^5) + 
   2 \beta (5 + 45 \gamma + 144 \gamma^2 + 224 \gamma^3 + 180 \gamma^4 + 
   60 \gamma^5)\Big]^3$

\bigskip 
\noindent
Because all the coefficients in both the numerator  and denominator of this expression 
are positive, the second derivative of ${\mathcal A}$ along any line $\beta+\gamma=\mbox{\bf const}$ is strictly positive whenever  $\beta, \gamma > 0$.  This proves the claim. 
\end{proof}

\begin{lem}\label{symmetry2}
 If $(\beta , \gamma )\in \RR^+\times \RR^+\approx \check{\zap K}$
  is a critical point of ${\mathcal A}$,
then $\beta=\gamma$. 
\end{lem}
\begin{proof}
Since there is an automorphism of $\CP_2\# 2\overline{\CP}_2$ which 
interchanges the two exceptional divisors, we automatically have 
 ${\mathcal A}(\beta , \gamma )={\mathcal A}( \gamma , \beta  )$. 
 In particular, reflection across the diagonal sends critical points to critical points. 
Lemma \ref{convex2}, however, implies that there can  be at most one critical 
point on  any  segment $\beta+ \gamma = \mbox{\bf const}$,
$\beta , \gamma > 0$.  If $\beta\neq \gamma$ for some critical point,
we would therefore obtain a contradiction by considering the 
line segment joining $(\beta , \gamma )$ to $( \gamma , \beta  )$. 
Hence every critical point must belong to the diagonal. 
\end{proof}

Using the  the terminology of \cite{chenlebweb}, we have thus 
shown  that any critical point $\Omega\in {\zap K}$ of
${\mathcal A}$ must be a {\em bilaterally symmetric K\"ahler class}.

\bigskip 

To conclude our discussion, we now let $F: \RR^+\to \RR$ be the function
defined by 
$$F(\beta) = {\mathcal A}(\beta , \beta).$$
Explicitly, this function is given by 
$$
F(\beta) =\frac{9 + 96 \beta + 396 \beta^2 + 840 \beta^3 + 954 \beta^4 + 
528 \beta^5 + 96 \beta^6}{1 + 
 12 \beta + 54 \beta^2 + 120 \beta^3 + 138 \beta^4 + 72 \beta^5 + 12 \beta^6}
$$
Because ${\mathcal A}$ is symmetric under $(\beta , \gamma )\leftrightarrow ( \gamma , \beta )$,
Lemma \ref{symmetry2} tells us that 
the critical points of ${\mathcal A}$ are precisely those  $(\beta , \beta )$ 
for which $F^\prime (\beta )=0$. 

\begin{lem} \label{prime2}
The above function satisfies $F^\prime (\beta ) > 0$ for all $\beta\geq 1.2$. 
\end{lem}
\begin{proof}
The derivative of $F$ is  given by 
$$\frac{dF}{d\beta}=  \frac{12 ~P(\beta)}{(1 + 12 \beta + 54 \beta^2 + 
  120 \beta^3 + 138 \beta^4 + 72 \beta^5 + 12 \beta^6)^2}$$
  where  the polynomial 
  $$P(\beta) :=
  -1 - 15 \beta - 96 \beta^2 - 336 \beta^3 - 680 \beta^4 - 720 \beta^5 - 120 \beta^6 + 
   624 \beta^7 + 708 \beta^8 + 300 \beta^9 + 48 \beta^{10}
  $$
must of course  always 
 have the same sign as  $F^\prime (\beta )$. When $\beta > 1$, however, we obviously have
 \begin{eqnarray*}
 P(\beta )& >&  -(1 + 15  + 96 + 336  + 680  + 720  + 120) \beta^6 + 
   (624  + 708  + 300  + 48) \beta^7\\
    &=& \beta^6 (1680 ~\beta - 1968)
\end{eqnarray*}
 $$
   $$
  and the claim therefore follows from the fact that $1.2 > 1968/1680$. 
\end{proof}

\begin{lem}  \label{doubleprime2}
The above function satisfies 
$F^{\prime\prime} (\beta ) > 0$ for all $\beta \in (0, 1.2]$. 
\end{lem}
\begin{proof}
The second derivative of $F$ is given by 
$$
\frac{d^2F}{d\beta^2}=  \frac{12 ~Q(\beta)}{(1 + 12 \beta + 54 \beta^2 + 
  120 \beta^3 + 138 \beta^4 + 72 \beta^5 + 12 \beta^6)^3}
$$
where the polynomial 
\begin{eqnarray*}
Q(\beta )& =& 9 + 204 \beta + 2142 \beta^2 + 13720 \beta^3 + 59514 \beta^4 + 183672 \beta^5 
 \\&& + 
 412044 \beta^6 +
  672768 \beta^7 + 782892 \beta^8 + 611088 \beta^9 + 264456 \beta^{10} \\&&- 
 2592 \beta^{11}  - 74952 \beta^{12} - 42336 \beta^{13} - 10800 \beta^{14} - 1152 \beta^{15}
\end{eqnarray*}
again has the same sign as $F^{\prime\prime} (\beta )$ in the allowed range $\beta > 0$. 
For $\beta \in (0,1]$, we thus have 
\begin{eqnarray*}
Q(\beta ) &\geq&(9 + 204  + 2142 + 13720  + 59514  + 183672  + 
 412044  +  672768  + 782892  + \\&& 611088  + 264456) \beta^{10} - 
 (2592 + 74952  + 42336  + 10800  + 1152 )\beta^{11}
 \\&=& (3002509 - 131832~\beta) \beta^{10}
\end{eqnarray*}
so that $F^{\prime\prime}(\beta )> 0$ in this range. Similarly, when $\beta > 1$ we have  
\begin{eqnarray*}
Q(\beta ) &> &(9 + 204  + 2142 + 13720  + 59514  + 183672  + 
 412044  +  672768  + 782892  + \\&& 611088  + 264456)  - 
 (2592 + 74952  + 42336  + 10800  + 1152 )\beta^{15}
 \\&=& 3002509 - 131832~\beta^{15}
\end{eqnarray*}
 and so, since $(1.2)^{15} < 16<  30.02509/1.31832$, we also conclude that 
  $F^{\prime\prime}(\beta )>0$ for $\beta \in (1,1.2]$. Combining these
  two arguments now proves the claim. 
  \end{proof}

  \begin{prop} \label{laudate} 
  For $M= \CP_2 \# 2 \overline{\CP}_2$, the 
   function ${\mathcal A}: \check{\zap K} \to \RR$ has exactly one 
   critical point. Moreover, this critical point is an   absolute minimum.
     \end{prop}
  \begin{proof} By Lemma \ref{symmetry2}, any critical point belongs
  to the line $\beta=\gamma$. Moreover, Lemma \ref{prime2} 
  says that such a critical point would necessarily satisfy $\beta \in (0,1.2]$.
  However, $F^\prime$ can only have one zero on $(0,1.2]$ by Lemma \ref{doubleprime2},
  and the uniqueness of the critical point is therefore assured. 
  
  The fact that 
  such a critical point does  actually exist  follows from the observation 
  that $\lim_{\beta\to 0}F^\prime (\beta ) = -12 < 0$, whereas 
  $F^{\prime}(1.2) > 0$ by Lemma \ref{prime2}. Moreover, since
  this unique critical point is a local minimum, it must actually be 
  the absolute minimum of $F$ on $(0,\infty )$.  However, 
    Lemma \ref{convex2} implies that 
  $${\mathcal A}(\beta , \gamma ) = \frac{{\mathcal A}(\beta , \gamma )+ {\mathcal A}( \gamma , \beta )}{2} 
   \geq {\mathcal A}(\frac{\beta + \gamma}{2} ,  \frac{\beta + \gamma}{2})= F (\frac{\beta + \gamma}{2})$$
  for any $\beta, \gamma > 0$. Thus  the absolute minimum of $F$   necessarily
 also represents the absolute minimum of ${\mathcal A}$.
  \end{proof}
  
  Theorem \ref{uni2} now follows. Indeed, 
   by Proposition \ref{aarhus}, 
  any conformally Einstein K\"ahler metric
  $\tilde{g}$ on  $M=\CP\# 2 \overline{\CP}_2$ is extremal, 
  and, 
  by Proposition \ref{critical}, 
  it must moreover  belong to a
  K\"ahler class $\Omega$ which is a critical point 
  of ${\mathcal A}$. However, by Proposition \ref{laudate}, 
   this critical K\"ahler class $\Omega$ is
  unique up to scale. Consequently,  the K\"ahler class of some 
  multiple of $\tilde{g}$ must coincide 
  with the K\"ahler class of 
   the conformally Einstein K\"ahler metric
  $g$ constructed in \cite{chenlebweb}. 
  On the other hand,  extremal K\"ahler metrics in any  
  fixed K\"ahler class are known  \cite{xxgang2} to be unique up to complex automorphisms;
  cf. \cite{donaldsonk1,mabuniq}. In particular, $(M,g)$ must be isometric 
  to $(M, c\tilde{g})$ for some positive constant $c$. 
  Moreover, since $\Omega$ actually minimizes ${\mathcal A}$
  by Proposition n \ref{laudate}, inequality (\ref{sharp}) implies that 
  $g$ actually minimizes ${\mathcal C}$ among all K\"ahler metrics on 
  $M=\CP\# 2 \overline{\CP}_2$.

\section{The case of $\CP\# 3 \overline{\CP}_2$} \label{comp3}

We now show that 
$${\mathcal A}: \check{\zap K}\to \RR$$
also  has only one critical point
  when  $M=\CP\# 3 \overline{\CP}_2$. 
First recall  that the general K\"ahler class
on this manifold is determined by four real numbers:

 
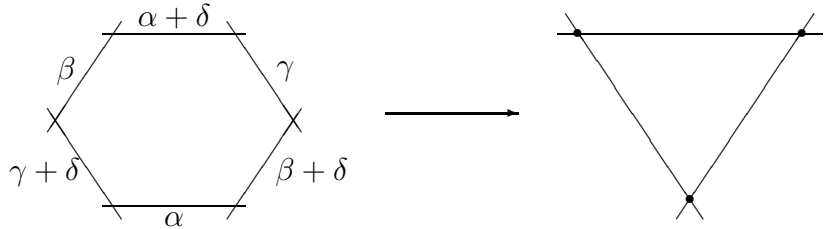
\begin{center}
\begin{picture}(240,80)(0,3) 
\put(-7,70){\line(1,0){54}}
\put(40,75){\line(2,-3){28}}
\put(0,75){\line(-2,-3){28}}
\put(-7,5){\line(1,0){54}}
\put(40,0){\line(2,3){28}}
\put(0,0){\line(-2,3){28}}
\put(20,0){\makebox(0,0){$\alpha$}}
\put(20,77){\makebox(0,0){$\alpha+\delta$}}
\put(-21,56){\makebox(0,0){$\beta$}}
\put(62,56){\makebox(0,0){$\gamma$}}
\put(-29,18){\makebox(0,0){$\gamma+\delta$}}
\put(72,18){\makebox(0,0){$\beta+\delta$}}
\put(100,40){\vector(1,0){50}}
\put(210,0){\line(2,3){52}}
\put(220,0){\line(-2,3){52}}
\put(165,70){\line(1,0){100}}
\put(215,7.5){\circle*{3}}
\put(172.6,70.5){\circle*{3}}
\put(257.4,70.5){\circle*{3}}
\end{picture}
\end{center}
Let us decompose the reduced K\"ahler cone $\check{\zap K}={\zap K}/\RR^+$ into
the open set ${\zap U}$ corresponding to $\delta > 0$, a second open set ${\zap U}^\prime$
corresponding to $\delta < 0$, and a lower-dimensional interface ${\zap P}$ 
cut out by the hyper-plane $\delta=0$. Each element of ${\zap U}$
then has a unique representative with $\delta =1$, and the remaining
positive real numbers $(\alpha, \beta, \gamma)$ then provide global coordinates 
on ${\zap U}$, which is thereby identified with $(\RR^+)^3$. The region 
${\zap U}^\prime$ is actually the image of ${\zap U}$ under a Cremona 
transformation $\Phi$, where $\Phi$ is the automorphism of  of $\CP\# 3 \overline{\CP}_2$
induced by the 
bimeromorphic transformation 
$$
\left[ z_1 : z_2 : z_3 \right] \longmapsto \left[ \frac{1}{z_1} : \frac{1}{z_2} : \frac{1}{z_3} \right]
$$
of $\CP_2$. Since $\mathcal A$ is invariant under automorphisms, 
it follows that we can completely understand its behavior on ${\zap U}^\prime$
by thoroughly understanding its behavior  on our coordinate domain ${\zap U}$. 
The behavior of $\mathcal A$ on the
interface ${\zap P}$ will of course call for a separate, careful  discussion. 

 For the present, however, let us  focus on our coordinate region 
 ${\zap U}$ in the reduced K\"ahler cone $\check{\zap K}$. 
 We now once again fix the $2$-torus in the automorphism group 
 corresponding to $[z_1:z_2:z_3]\mapsto [e^{i\theta}z_1:e^{i\phi}z_2:z_3]$.
 The image of $M$ under the moment map is then the hexagon 
 
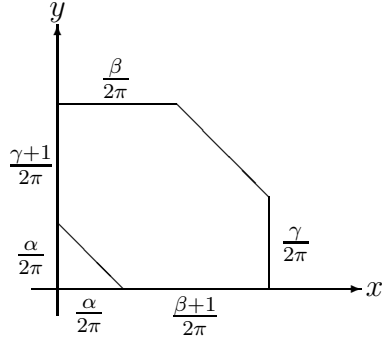
\begin{center}
\begin{picture}(120,120)(0,0)
\put(0,10){\vector(1,0){125}}
\put(10,0){\vector(0,1){110}}
\put(130,10){\makebox(0,0){$x$}}
\put(10,115){\makebox(0,0){$y$}}
\put(22,0){\makebox(0,0){$\frac{\alpha}{2\pi}$}}
\put(0,22){\makebox(0,0){$\frac{\alpha}{2\pi}$}}
\put(62,0){\makebox(0,0){$\frac{\beta +1}{2\pi}$}}
\put(0,57){\makebox(0,0){$\frac{\gamma +1}{2\pi}$}}
\put(100,27){\makebox(0,0){$\frac{\gamma }{2\pi}$}}
\put(32,90){\makebox(0,0){$\frac{\beta }{2\pi}$}}
\put(10,35){\line(1,-1){25}}
\put(10,80){\line(1,0){45}}
\put(55,80){\line(1,-1){35}}
\put(90,10){\line(0,1){35}}
\end{picture}
\end{center}
and our formulas \cite{ls} for the components of the Futaki invariant become 
\begin{eqnarray*}
{\mathfrak F}_1&=& \int_M x(s-s_0) d\mu \\
&=& \frac{1}{V}\left[ 
(\alpha+\beta -2\gamma) (\frac{1}{3} + \gamma + \gamma^2) + (\gamma-\alpha) (\gamma - \beta) (2+\alpha+ \beta + 2\gamma ) 
\right]\\
{\mathfrak F}_2&=& \int_M y(s-s_0) d\mu \\
&=& \frac{1}{V}\left[ 
(\alpha+ \gamma -2\beta) (\frac{1}{3} + \beta + \beta^2) + (\beta-\alpha) (\beta - \gamma) (2+ \alpha+\gamma + 2\beta ) 
\right]\\
\end{eqnarray*}
where 
$$
V= \alpha \beta+ \alpha \gamma + \beta \gamma + \alpha + \beta + \gamma +  \frac{1}{2}
$$
is the volume of $(M,\Omega)$. 
Three other essential coefficients needed in our computation are 
\begin{eqnarray*}
A&:=& \int_M (x-x_0)^2d\mu
 \\  &=& (288\pi^2 V)^{-1}
   \Big[1 + 6 \beta + 12 \beta^2 + 12 \beta^3 + 6 \beta^4 + 
  6 \gamma^2 (1 + \beta)^4 + 
  6 \alpha^4 (1 + \gamma + \beta)^2 +
    \\&& 6 \gamma (1 + 5 \beta + 8 \beta^2 + 6 \beta^3 + 2 \beta^4) + 
  6 \alpha^2 (2 + 8 \beta + 9 \beta^2 + 4 \beta^3 + \beta^4 + 6 \gamma^2 (1 + \beta)^2 + 
\\&&     2 \gamma (2 + \beta)^2 (1 + 2 \beta)) + 
  12 \alpha^3 (1 + 3 \beta + 2 \beta^2 + 2 \gamma^2 (1 + \beta) + \gamma (3 + 6 \beta + 2 \beta^2)) + 
\\&&   6 \alpha (1 + 5 \beta + 8 \beta^2 + 6 \beta^3 + 2 \beta^4 + 4 \gamma^2 (1 + \beta)^3 + 
     \gamma (5 + 20 \beta + 24 \beta^2 + 12 \beta^3 + 2 \beta^4))\Big]
  \end{eqnarray*}
 \begin{eqnarray*} 
      B&:=&
    \int_M (y-y_0)^2d\mu\\
   &=& (288\pi^2 V)^{-1}
   \Big[1 + 6 \gamma + 12 \gamma^2 + 12 \gamma^3 + 6 \gamma^4 + 
  6 \beta^2 (1 + \gamma)^4 + 
  6 \alpha^4 (1 + \beta + \gamma)^2 +
    \\&& 6 \beta (1 + 5 \gamma + 8 \gamma^2 + 6 \gamma^3 + 2 \gamma^4) + 
  6 \alpha^2 (2 + 8 \gamma + 9 \gamma^2 + 4 \gamma^3 + \gamma^4 + 6 \beta^2 (1 + \gamma)^2 + 
\\&&     2 \beta (2 + \gamma)^2 (1 + 2 \gamma)) + 
  12 \alpha^3 (1 + 3 \gamma + 2 \gamma^2 + 2 \beta^2 (1 + \gamma) + \beta (3 + 6 \gamma + 2 \gamma^2)) + 
\\&&   6 \alpha (1 + 5 \gamma + 8 \gamma^2 + 6 \gamma^3 + 2 \gamma^4 + 4 \beta^2 (1 + \gamma)^3 + 
     \beta (5 + 20 \gamma + 24 \gamma^2 + 12 \gamma^3 + 2 \gamma^4))\Big]
\end{eqnarray*}   
  and 
 \begin{eqnarray*} 
   C &:=&  \int_M (x-x_0)(y-y_0)d\mu\\
 &=& -(576\pi^2 V)^{-1} \Big[1 + 6 \gamma + 6 \gamma^2 + 12 \alpha^4 (1 + \beta + \gamma)^2 +
  6 \beta^2 (1 + 4 \gamma + 3 \gamma^2) + 
\\&&  6 \beta (1 + 5 \gamma + 4 \gamma^2) + 
 24 \alpha^3 (1 + 3 \gamma + 2 \gamma^2 + 2 \beta^2 (1 + \gamma) + \beta (3 + 6 \gamma + 2 \gamma^2)) + 
 \\&& 18 \alpha^2 (1 + 4 \gamma + 3 \gamma^2 + \beta^2 (3 + 6 \gamma + 2 \gamma^2) + 
    2 \beta (2 + 6 \gamma + 3 \gamma^2)) + 
\\&&  6 \alpha (1 + 5 \gamma + 4 \gamma^2 + 2 \beta^2 (2 + 6 \gamma + 3 \gamma^2) + 
    \beta (5 + 20 \gamma + 12 \gamma^2))\Big]
\end{eqnarray*}
We then have
$$
\| {\mathfrak F}\|^2 = \frac{B{\mathfrak F}_1^2 - 2C{\mathfrak F}_1{\mathfrak F}_+A{\mathfrak F}_2^2 }{AB-C^2}~.
$$
This in turn allows  us to compute 
$${\mathcal A} = \frac{(c_1\cdot \Omega)^2}{\Omega^2}+ \frac{1}{32\pi^2}  \frac{B{\mathfrak F}_1^2 - 2C{\mathfrak F}_1{\mathfrak F}_+A{\mathfrak F}_2^2 }{AB-C^2}~,$$
which, upon simplification, is explicitly given by 
\bigskip

{\footnotesize 
\noindent
$\displaystyle
3 \Big[ 3 + 28 \gamma + 96 \gamma^2 + 168 \gamma^3 + 164 \gamma^4 + 80 \gamma^5 + 16 \gamma^6 + 
     16 \beta^6 (1 + \gamma)^4 + 16 \alpha^6 (1 + \beta + \gamma)^4 + 
     16 \beta^5 (5 + 24 \gamma + 43 \gamma^2 + 37 \gamma^3 + 15 \gamma^4 + 2 \gamma^5) + 
     4 \beta^4 (41 + 228 \gamma + 478 \gamma^2 + 496 \gamma^3 + 263 \gamma^4 + 60 \gamma^5 + 
        4 \gamma^6) + 
     8 \beta^3 (21 + 135 \gamma + 326 \gamma^2 + 392 \gamma^3 + 248 \gamma^4 + 74 \gamma^5 + 
        8 \gamma^6) + 
     4 \beta (7 + 58 \gamma + 176 \gamma^2 + 270 \gamma^3 + 228 \gamma^4 + 96 \gamma^5 + 16 \gamma^6) + 
     4 \beta^2 (24 + 176 \gamma + 479 \gamma^2 + 652 \gamma^3 + 478 \gamma^4 + 172 \gamma^5 + 
        24 \gamma^6) + 
     16 \alpha^5 (5 + 2 \beta^5 + 24 \gamma + 43 \gamma^2 + 37 \gamma^3 + 15 \gamma^4 + 2 \gamma^5 + 
        \beta^4 (15 + 14 \gamma) + \beta^3 (37 + 70 \gamma + 30 \gamma^2) + 
        \beta^2 (43 + 123 \gamma + 108 \gamma^2 + 30 \gamma^3) + 
        \beta (24 + 92 \gamma + 123 \gamma^2 + 70 \gamma^3 + 14 \gamma^4)) + 
     4 \alpha^4 (41 + 4 \beta^6 + 228 \gamma + 478 \gamma^2 + 496 \gamma^3 + 263 \gamma^4 + 
        60 \gamma^5 + 4 \gamma^6 + \beta^5 (60 + 56 \gamma) + 
        \beta^4 (263 + 476 \gamma + 196 \gamma^2) + 
        8 \beta^3 (62 + 169 \gamma + 139 \gamma^2 + 35 \gamma^3) + 
        2 \beta^2 (239 + 876 \gamma + 1089 \gamma^2 + 556 \gamma^3 + 98 \gamma^4) + 
        4 \beta (57 + 263 \gamma + 438 \gamma^2 + 338 \gamma^3 + 119 \gamma^4 + 14 \gamma^5)) + 
     8 \alpha^3 (21 + 135 \gamma + 326 \gamma^2 + 392 \gamma^3 + 248 \gamma^4 + 74 \gamma^5 + 
        8 \gamma^6 + 8 \beta^6 (1 + \gamma) + 2 \beta^5 (37 + 70 \gamma + 30 \gamma^2) + 
        4 \beta^4 (62 + 169 \gamma + 139 \gamma^2 + 35 \gamma^3) + 
        4 \beta^3 (98 + 353 \gamma + 428 \gamma^2 + 210 \gamma^3 + 35 \gamma^4) + 
        2 \beta^2 (163 + 735 \gamma + 1179 \gamma^2 + 856 \gamma^3 + 278 \gamma^4 + 30 \gamma^5) + 
        \beta (135 + 736 \gamma + 1470 \gamma^2 + 1412 \gamma^3 + 676 \gamma^4 + 140 \gamma^5 + 
           8 \gamma^6)) + 
     4 \alpha (7 + 58 \gamma + 176 \gamma^2 + 270 \gamma^3 + 228 \gamma^4 + 96 \gamma^5 + 16 \gamma^6 + 
        16 \beta^6 (1 + \gamma)^3 + 
        4 \beta^5 (24 + 92 \gamma + 123 \gamma^2 + 70 \gamma^3 + 14 \gamma^4) + 
        4 \beta^4 (57 + 263 \gamma + 438 \gamma^2 + 338 \gamma^3 + 119 \gamma^4 + 14 \gamma^5) + 
        2 \beta^3 (135 + 736 \gamma + 1470 \gamma^2 + 1412 \gamma^3 + 676 \gamma^4 + 
           140 \gamma^5 + 8 \gamma^6) + 
        4 \beta^2 (44 + 278 \gamma + 645 \gamma^2 + 735 \gamma^3 + 438 \gamma^4 + 123 \gamma^5 + 
           12 \gamma^6) + 
        2 \beta (29 + 210 \gamma + 556 \gamma^2 + 736 \gamma^3 + 526 \gamma^4 + 184 \gamma^5 + 
           24 \gamma^6)) + 
     4 \alpha^2 (24 + 176 \gamma + 479 \gamma^2 + 652 \gamma^3 + 478 \gamma^4 + 172 \gamma^5 + 
        24 \gamma^6 + 24 \beta^6 (1 + \gamma)^2 + 
        4 \beta^5 (43 + 123 \gamma + 108 \gamma^2 + 30 \gamma^3) + 
        2 \beta^4 (239 + 876 \gamma + 1089 \gamma^2 + 556 \gamma^3 + 98 \gamma^4) + 
        4 \beta^3 (163 + 735 \gamma + 1179 \gamma^2 + 856 \gamma^3 + 278 \gamma^4 + 30 \gamma^5) + 
        4 \beta (44 + 278 \gamma + 645 \gamma^2 + 735 \gamma^3 + 438 \gamma^4 + 123 \gamma^5 + 
           12 \gamma^6) + 
        \beta^2 (479 + 2580 \gamma + 5058 \gamma^2 + 4716 \gamma^3 + 2178 \gamma^4 + 
           432 \gamma^5 + 24 \gamma^6))\Big]\Big/ \\
           \Big[ 1 + 10 \gamma + 36 \gamma^2 + 64 \gamma^3 + 
   60 \gamma^4 + 24 \gamma^5 + 24 \beta^5 (1 + \gamma)^5 + 24 \alpha^5 (1 + \beta + \gamma)^5 + 
   12 \beta^4 (1 + \gamma)^2 (5 + 20 \gamma + 23 \gamma^2 + 10 \gamma^3) + 
   16 \beta^3 (4 + 28 \gamma + 72 \gamma^2 + 90 \gamma^3 + 57 \gamma^4 + 15 \gamma^5) + 
   12 \beta^2 (3 + 24 \gamma + 69 \gamma^2 + 96 \gamma^3 + 68 \gamma^4 + 20 \gamma^5) + 
   2 \beta (5 + 45 \gamma + 144 \gamma^2 + 224 \gamma^3 + 180 \gamma^4 + 60 \gamma^5) + 
   12 \alpha^4 (1 + \beta + \gamma)^2 (5 + 20 \gamma + 23 \gamma^2 + 10 \gamma^3 + 10 \beta^3 (1 + \gamma) +
       \beta^2 (23 + 46 \gamma + 16 \gamma^2) + 2 \beta (10 + 30 \gamma + 23 \gamma^2 + 5 \gamma^3)) + 
   16 \alpha^3 (4 + 28 \gamma + 72 \gamma^2 + 90 \gamma^3 + 57 \gamma^4 + 15 \gamma^5 + 
      15 \beta^5 (1 + \gamma)^2 + 3 \beta^4 (19 + 57 \gamma + 50 \gamma^2 + 13 \gamma^3) + 
      3 \beta^3 (30 + 120 \gamma + 155 \gamma^2 + 78 \gamma^3 + 13 \gamma^4) + 
      3 \beta^2 (24 + 120 \gamma + 206 \gamma^2 + 155 \gamma^3 + 50 \gamma^4 + 5 \gamma^5) + 
      \beta (28 + 168 \gamma + 360 \gamma^2 + 360 \gamma^3 + 171 \gamma^4 + 30 \gamma^5)) + 
   12 \alpha^2 (3 + 24 \gamma + 69 \gamma^2 + 96 \gamma^3 + 68 \gamma^4 + 20 \gamma^5 + 
      20 \beta^5 (1 + \gamma)^3 + 
      \beta^4 (68 + 272 \gamma + 366 \gamma^2 + 200 \gamma^3 + 36 \gamma^4) + 
      4 \beta^3 (24 + 120 \gamma + 206 \gamma^2 + 155 \gamma^3 + 50 \gamma^4 + 5 \gamma^5) + 
      2 \beta (12 + 84 \gamma + 207 \gamma^2 + 240 \gamma^3 + 136 \gamma^4 + 30 \gamma^5) + 
      \beta^2 (69 + 414 \gamma + 864 \gamma^2 + 824 \gamma^3 + 366 \gamma^4 + 60 \gamma^5)) + 
   2 \alpha (5 + 45 \gamma + 144 \gamma^2 + 224 \gamma^3 + 180 \gamma^4 + 60 \gamma^5 + 
      60 \beta^5 (1 + \gamma)^4 + 
      12 \beta^4 (15 + 75 \gamma + 136 \gamma^2 + 114 \gamma^3 + 43 \gamma^4 + 5 \gamma^5) + 
      12 \beta^2 (12 + 84 \gamma + 207 \gamma^2 + 240 \gamma^3 + 136 \gamma^4 + 30 \gamma^5) + 
      8 \beta^3 (28 + 168 \gamma + 360 \gamma^2 + 360 \gamma^3 + 171 \gamma^4 + 30 \gamma^5) + 
      3 \beta (15 + 120 \gamma + 336 \gamma^2 + 448 \gamma^3 + 300 \gamma^4 + 80 \gamma^5))\Big]
$}

\begin{lem} \label{convex3} 
The function obtained by restricting ${\mathcal A}$ to any line of the form 
$\alpha + \beta = \mbox{\bf const}$,  $\gamma = \widehat{\mbox{\bf const}}$
is strictly convex on the segment  $\alpha, \beta  > 0$. 
\end{lem}
\begin{proof} Mechanically computing $\displaystyle \left(\frac{\partial}{\partial \alpha} -\frac{\partial}{\partial \beta}\right)^2 {\mathcal A}$, we obtain

{\footnotesize
\noindent
$\displaystyle 
12 \Big[ 4608 \beta^{16} (1 + \gamma)^{12} + 4608 \alpha^{16} (1 + \beta + \gamma)^{12} + 
     4608 \beta^{15} (1 + \gamma)^9 (10 + 39 \gamma + 48 \gamma^2 + 25 \gamma^3 + 3 \gamma^4) + 
     1152 \beta^{14} (1 + \gamma)^6 (197 + 1542 \gamma + 4933 \gamma^2 + 8484 \gamma^3 + 
        8617 \gamma^4 + 5190 \gamma^5 + 1721 \gamma^6 + 252 \gamma^7 + 12 \gamma^8) + (1 + 
        4 \gamma + 6 \gamma^2 + 4 \gamma^3)^2 (7 + 146 \gamma + 1240 \gamma^2 + 5568 \gamma^3 + 
        14520 \gamma^4 + 22992 \gamma^5 + 22176 \gamma^6 + 12288 \gamma^7 + 3312 \gamma^8 + 
        288 \gamma^9) + 
     384 \beta^{13} (1 + \gamma)^3 (1861 + 21942 \gamma + 113487 \gamma^2 + 341094 \gamma^3 + 
        664392 \gamma^4 + 882156 \gamma^5 + 813518 \gamma^6 + 518262 \gamma^7 + 
        221247 \gamma^8 + 59358 \gamma^9 + 8895 \gamma^{10} + 612 \gamma^{11} + 12 \gamma^{12}) + 
     192 \beta^{12} (1 + \gamma)^2 (8250 + 113750 \gamma + 696062 \gamma^2 + 2508168 \gamma^3 + 
        5950477 \gamma^4 + 9822592 \gamma^5 + 11586899 \gamma^6 + 9839926 \gamma^7 + 
        5959811 \gamma^8 + 2503386 \gamma^9 + 691353 \gamma^{10} + 114366 \gamma^{11} + 
        9684 \gamma^{12} + 312 \gamma^{13}) + 
     192 \beta^{11} (13533 + 227527 \gamma + 1729020 \gamma^2 + 7895570 \gamma^3 + 
        24283812 \gamma^4 + 53371034 \gamma^5 + 86645330 \gamma^6 + 105734746 \gamma^7 + 
        97570741 \gamma^8 + 67802847 \gamma^9 + 34957058 \gamma^{10} + 12997772 \gamma^{11} + 
        3323034 \gamma^{12} + 539904 \gamma^{13} + 48744 \gamma^{14} + 1872 \gamma^{15}) + 
     2 \beta (83 + 2301 \gamma + 28494 \gamma^2 + 209856 \gamma^3 + 1033632 \gamma^4 + 
        3628668 \gamma^5 + 9431904 \gamma^6 + 18562128 \gamma^7 + 27957616 \gamma^8 + 
        32247792 \gamma^9 + 28207200 \gamma^{10} + 18287232 \gamma^{11} + 8424576 \gamma^{12} + 
        2555712 \gamma^{13} + 440064 \gamma^{14} + 29952 \gamma^{15}) + 
     32 \beta^{10} (101571 + 1813896 \gamma + 14604432 \gamma^2 + 70513188 \gamma^3 + 
        228947238 \gamma^4 + 530714796 \gamma^5 + 908577086 \gamma^6 + 
        1170005160 \gamma^7 + 1141232271 \gamma^8 + 840744108 \gamma^9 + 
        461642418 \gamma^{10} + 184071084 \gamma^{11} + 50987928 \gamma^{12} + 
        9116424 \gamma^{13} + 926928 \gamma^{14} + 41184 \gamma^{15}) + 
     32 \beta^9 (98661 + 1866687 \gamma + 15885618 \gamma^2 + 80906116 \gamma^3 + 
        276675866 \gamma^4 + 674840994 \gamma^5 + 1215258340 \gamma^6 + 
        1646801188 \gamma^7 + 1692453549 \gamma^8 + 1316646147 \gamma^9 + 
        766098846 \gamma^{10} + 325359792 \gamma^{11} + 96708696 \gamma^{12} + 
        18750876 \gamma^{13} + 2095344 \gamma^{14} + 102960 \gamma^{15}) + 
     2 \beta^2 (945 + 25140 \gamma + 298662 \gamma^2 + 2109468 \gamma^3 + 9958968 \gamma^4 + 
        33486216 \gamma^5 + 83285952 \gamma^6 + 156665280 \gamma^7 + 225273008 \gamma^8 + 
        247789632 \gamma^9 + 206497248 \gamma^{10} + 127489728 \gamma^{11} + 
        55960704 \gamma^{12} + 16219008 \gamma^{13} + 2688768 \gamma^{14} + 179712 \gamma^{15}) + 
     32 \beta^7 (44965 + 947906 \gamma + 8953208 \gamma^2 + 50433780 \gamma^3 + 
        190237205 \gamma^4 + 510946738 \gamma^5 + 1012667298 \gamma^6 + 
        1511386522 \gamma^7 + 1714215562 \gamma^8 + 1476815916 \gamma^9 + 
        956433924 \gamma^{10} + 455317224 \gamma^{11} + 153148944 \gamma^{12} + 
        34002144 \gamma^{13} + 4396896 \gamma^{14} + 247104 \gamma^{15}) + 
     4 \beta^3 (3375 + 86033 \gamma + 979230 \gamma^2 + 6625272 \gamma^3 + 
        29950718 \gamma^4 + 96368312 \gamma^5 + 229144480 \gamma^6 + 411607600 \gamma^7 + 
        564473648 \gamma^8 + 591405360 \gamma^9 + 468905088 \gamma^{10} + 
        275222400 \gamma^{11} + 114864288 \gamma^{12} + 31720320 \gamma^{13} + 
        5044608 \gamma^{14} + 329472 \gamma^{15}) + 
     16 \beta^8 (150265 + 3004570 \gamma + 26965276 \gamma^2 + 144567028 \gamma^3 + 
        519661332 \gamma^4 + 1331111864 \gamma^5 + 2516570692 \gamma^6 + 
        3581417380 \gamma^7 + 3869506165 \gamma^8 + 3170504622 \gamma^9 + 
        1948355292 \gamma^{10} + 877378944 \gamma^{11} + 278036964 \gamma^{12} + 
        57887208 \gamma^{13} + 6997104 \gamma^{14} + 370656 \gamma^{15}) + 
     8 \beta^5 (30513 + 710253 \gamma + 7388664 \gamma^2 + 45729222 \gamma^3 + 
        189176247 \gamma^4 + 556770619 \gamma^5 + 1209448620 \gamma^6 + 
        1980822668 \gamma^7 + 2470651048 \gamma^8 + 2347662648 \gamma^9 + 
        1683260592 \gamma^{10} + 891089664 \gamma^{11} + 334901232 \gamma^{12} + 
        83403504 \gamma^{13} + 12075264 \gamma^{14} + 741312 \gamma^{15}) + 
     4 \beta^4 (16756 + 408558 \gamma + 4449084 \gamma^2 + 28805710 \gamma^3 + 
        124612737 \gamma^4 + 383532430 \gamma^5 + 871687260 \gamma^6 + 
        1495054424 \gamma^7 + 1955180872 \gamma^8 + 1950734256 \gamma^9 + 
        1470877824 \gamma^{10} + 820102368 \gamma^{11} + 324992016 \gamma^{12} + 
        85324896 \gamma^{13} + 12972096 \gamma^{14} + 823680 \gamma^{15}) + 
     8 \beta^6 (84231 + 1867880 \gamma + 18531878 \gamma^2 + 109501792 \gamma^3 + 
        432809387 \gamma^4 + 1217400396 \gamma^5 + 2526813284 \gamma^6 + 
        3951435944 \gamma^7 + 4700685728 \gamma^8 + 4253940816 \gamma^9 + 
        2899691520 \gamma^{10} + 1456519776 \gamma^{11} + 518405040 \gamma^{12} + 
        122122944 \gamma^{13} + 16759872 \gamma^{14} + 988416 \gamma^{15}) + 
     4608 \alpha^{15} (1 + \beta + \gamma)^9 (10 + \beta^4 + 39 \gamma + 48 \gamma^2 + 25 \gamma^3 + 
        3 \gamma^4 + \beta^3 (19 + 18 \gamma) + \beta^2 (42 + 81 \gamma + 30 \gamma^2) + 
        \beta (37 + 108 \gamma + 87 \gamma^2 + 22 \gamma^3)) + 
     1152 \alpha^{14} (1 + \beta + \gamma)^6 (197 + 1542 \gamma + 4933 \gamma^2 + 8484 \gamma^3 + 
        8617 \gamma^4 + 5190 \gamma^5 + 1721 \gamma^6 + 252 \gamma^7 + 12 \gamma^8 + 
        60 \beta^7 (1 + \gamma) + \beta^6 (809 + 1548 \gamma + 708 \gamma^2) + 
        6 \beta^5 (517 + 1493 \gamma + 1326 \gamma^2 + 358 \gamma^3) + 
        3 \beta^4 (1983 + 7686 \gamma + 10285 \gamma^2 + 5628 \gamma^3 + 1084 \gamma^4) + 
        4 \beta^3 (1629 + 7930 \gamma + 14256 \gamma^2 + 11881 \gamma^3 + 4641 \gamma^4 + 
           675 \gamma^5) + 
        \beta^2 (4141 + 24264 \gamma + 54882 \gamma^2 + 61752 \gamma^3 + 36759 \gamma^4 + 
           10836 \gamma^5 + 1212 \gamma^6) + 
        2 \beta (703 + 4813 \gamma + 13128 \gamma^2 + 18636 \gamma^3 + 14961 \gamma^4 + 
           6675 \gamma^5 + 1470 \gamma^6 + 114 \gamma^7)) + 
     384 \alpha^{13} (1 + \beta + \gamma)^3 (1861 + 12 \beta^{12} + 21942 \gamma + 113487 \gamma^2 + 
        341094 \gamma^3 + 664392 \gamma^4 + 882156 \gamma^5 + 813518 \gamma^6 + 
        518262 \gamma^7 + 221247 \gamma^8 + 59358 \gamma^9 + 8895 \gamma^{10} + 612 \gamma^{11} + 
        12 \gamma^{12} + 144 \beta^{11} (1 + \gamma) + 
        3 \beta^{10} (673 + 1356 \gamma + 672 \gamma^2) + 
        6 \beta^9 (3197 + 9356 \gamma + 8844 \gamma^2 + 2704 \gamma^3) + 
        9 \beta^8 (10065 + 39024 \gamma + 54321 \gamma^2 + 32156 \gamma^3 + 6828 \gamma^4) + 
        6 \beta^7 (41603 + 201823 \gamma + 373269 \gamma^2 + 328818 \gamma^3 + 
           138228 \gamma^4 + 22272 \gamma^5) + 
        6 \beta^6 (73962 + 431749 \gamma + 999492 \gamma^2 + 1174635 \gamma^3 + 
           740784 \gamma^4 + 238344 \gamma^5 + 30588 \gamma^6) + 
        6 \beta^5 (89048 + 608214 \gamma + 1694211 \gamma^2 + 2496921 \gamma^3 + 
           2107647 \gamma^4 + 1020564 \gamma^5 + 262284 \gamma^6 + 27504 \gamma^7) + 
        6 \beta^4 (73701 + 576761 \gamma + 1879790 \gamma^2 + 3337291 \gamma^3 + 
           3538190 \gamma^4 + 2295699 \gamma^5 + 888168 \gamma^6 + 186312 \gamma^7 + 
           16074 \gamma^8) + 
        2 \beta^3 (124247 + 1096029 \gamma + 4093467 \gamma^2 + 8510758 \gamma^3 + 
           10880079 \gamma^4 + 8870427 \gamma^5 + 4595697 \gamma^6 + 1447758 \gamma^7 + 
           248580 \gamma^8 + 17400 \gamma^9) + 
        3 \beta^2 (30223 + 296646 \gamma + 1249428 \gamma^2 + 2979674 \gamma^3 + 
           4465026 \gamma^4 + 4390950 \gamma^5 + 2856444 \gamma^6 + 1201770 \gamma^7 + 
           307863 \gamma^8 + 42324 \gamma^9 + 2280 \gamma^{10}) + 
        6 \beta (3235 + 34956 \gamma + 163965 \gamma^2 + 441579 \gamma^3 + 760003 \gamma^4 + 
           876826 \gamma^5 + 688049 \gamma^6 + 362935 \gamma^7 + 123588 \gamma^8 + 
           25052 \gamma^9 + 2592 \gamma^{10} + 96 \gamma^{11})) + 
     192 \alpha^{12} (1 + \beta + \gamma)^2 (8250 + 24 \beta^{14} + 113750 \gamma + 696062 \gamma^2 + 
        2508168 \gamma^3 + 5950477 \gamma^4 + 9822592 \gamma^5 + 11586899 \gamma^6 + 
        9839926 \gamma^7 + 5959811 \gamma^8 + 2503386 \gamma^9 + 691353 \gamma^{10} + 
        114366 \gamma^{11} + 9684 \gamma^{12} + 312 \gamma^{13} + 24 \beta^{13} (26 + 25 \gamma) + 
        6 \beta^{12} (957 + 1852 \gamma + 864 \gamma^2) + 
        6 \beta^{11} (6543 + 19317 \gamma + 18480 \gamma^2 + 5744 \gamma^3) + 
        3 \beta^{10} (75223 + 294322 \gamma + 418650 \gamma^2 + 257168 \gamma^3 + 
           57472 \gamma^4) + 
        6 \beta^9 (149670 + 727904 \gamma + 1364327 \gamma^2 + 1233511 \gamma^3 + 
           537848 \gamma^4 + 90348 \gamma^5) + 
        3 \beta^8 (797605 + 4648792 \gamma + 10831085 \gamma^2 + 12918942 \gamma^3 + 
           8326512 \gamma^4 + 2751528 \gamma^5 + 364344 \gamma^6) + 
        6 \beta^7 (732661 + 4987897 \gamma + 13925711 \gamma^2 + 20673338 \gamma^3 + 
           17645016 \gamma^4 + 8669796 \gamma^5 + 2271744 \gamma^6 + 244536 \gamma^7) + 
        3 \beta^6 (1907229 + 14872408 \gamma + 48472654 \gamma^2 + 86267730 \gamma^3 + 
           91820960 \gamma^4 + 59925768 \gamma^5 + 23420508 \gamma^6 + 5001792 \gamma^7 + 
           444912 \gamma^8) + 
        2 \beta^5 (2671576 + 23495460 \gamma + 87645408 \gamma^2 + 182095913 \gamma^3 + 
           232560561 \gamma^4 + 189526326 \gamma^5 + 98502876 \gamma^6 + 
           31393764 \gamma^7 + 5536080 \gamma^8 + 407652 \gamma^9) + 
        \beta^4 (3567167 + 34939894 \gamma + 146886663 \gamma^2 + 349261306 \gamma^3 + 
           520959122 \gamma^4 + 509714070 \gamma^5 + 330864888 \gamma^6 + 
           140138568 \gamma^7 + 36801918 \gamma^8 + 5349192 \gamma^9 + 
           320952 \gamma^{10}) + 
        2 \beta^3 (832542 + 8988286 \gamma + 42061729 \gamma^2 + 112720590 \gamma^3 + 
           192501233 \gamma^4 + 220052087 \gamma^5 + 171462999 \gamma^6 + 
           90598194 \gamma^7 + 31528251 \gamma^8 + 6786381 \gamma^9 + 795552 \gamma^{10} + 
           37080 \gamma^{11}) + 
        \beta^2 (517294 + 6102330 \gamma + 31470582 \gamma^2 + 93918498 \gamma^3 + 
           180890595 \gamma^4 + 236871504 \gamma^5 + 215639386 \gamma^6 + 
           136637922 \gamma^7 + 59158419 \gamma^8 + 16778898 \gamma^9 + 
           2882574 \gamma^{10} + 257904 \gamma^{11} + 8376 \gamma^{12}) + 
        2 \beta (48254 + 617382 \gamma + 3480117 \gamma^2 + 11456064 \gamma^3 + 
           24599051 \gamma^4 + 36365274 \gamma^5 + 37946148 \gamma^6 + 
           28087545 \gamma^7 + 14560881 \gamma^8 + 5116890 \gamma^9 + 1147983 \gamma^{10} + 
           148095 \gamma^{11} + 9024 \gamma^{12} + 156 \gamma^{13})) + 
     192 \alpha^{11} (13533 + 227527 \gamma + 1729020 \gamma^2 + 7895570 \gamma^3 + 
        24283812 \gamma^4 + 53371034 \gamma^5 + 86645330 \gamma^6 + 105734746 \gamma^7 + 
        97570741 \gamma^8 + 67802847 \gamma^9 + 34957058 \gamma^{10} + 12997772 \gamma^{11} + 
        3323034 \gamma^{12} + 539904 \gamma^{13} + 48744 \gamma^{14} + 1872 \gamma^{15} + 
        288 \beta^{16} (1 + \gamma) + 72 \beta^{15} (83 + 162 \gamma + 76 \gamma^2) + 
        144 \beta^{14} (369 + 1075 \gamma + 1008 \gamma^2 + 304 \gamma^3) + 
        36 \beta^{13} (8574 + 33448 \gamma + 47437 \gamma^2 + 29008 \gamma^3 + 6462 \gamma^4) + 
        12 \beta^{12} (115718 + 564926 \gamma + 1070478 \gamma^2 + 985675 \gamma^3 + 
           441228 \gamma^4 + 76782 \gamma^5) + 
        6 \beta^{11} (813725 + 4760040 \gamma + 11241210 \gamma^2 + 13741020 \gamma^3 + 
           9174366 \gamma^4 + 3170808 \gamma^5 + 442680 \gamma^6) + 
        6 \beta^{10} (2177829 + 14844821 \gamma + 41930424 \gamma^2 + 63707682 \gamma^3 + 
           56273970 \gamma^4 + 28901574 \gamma^5 + 7986696 \gamma^6 + 914952 \gamma^7) + 
        \beta^9 (26373871 + 205456130 \gamma + 675804032 \gamma^2 + 
           1227179856 \gamma^3 + 1346713032 \gamma^4 + 915029208 \gamma^5 + 
           375865032 \gamma^6 + 85255200 \gamma^7 + 8158608 \gamma^8) + 
        3 \beta^8 (13427795 + 117807109 \gamma + 442670586 \gamma^2 + 
           935756704 \gamma^3 + 1227490494 \gamma^4 + 1036889712 \gamma^5 + 
           564058200 \gamma^6 + 190371336 \gamma^7 + 36093600 \gamma^8 + 
           2918064 \gamma^9) + 
        6 \beta^7 (7801383 + 76183516 \gamma + 322221824 \gamma^2 + 777806072 \gamma^3 + 
           1187803396 \gamma^4 + 1199952552 \gamma^5 + 812143884 \gamma^6 + 
           363251736 \gamma^7 + 102511560 \gamma^8 + 16414440 \gamma^9 + 
           1125264 \gamma^{10}) + 
        2 \beta^6 (20712775 + 222961725 \gamma + 1048949606 \gamma^2 + 
           2848718020 \gamma^3 + 4967195728 \gamma^4 + 5843028064 \gamma^5 + 
           4731024660 \gamma^6 + 2633298420 \gamma^7 + 984662748 \gamma^8 + 
           234477204 \gamma^9 + 31773744 \gamma^{10} + 1836288 \gamma^{11}) + 
        2 \beta^5 (13878853 + 163335480 \gamma + 846436347 \gamma^2 + 
           2555731368 \gamma^3 + 5012833029 \gamma^4 + 6733073982 \gamma^5 + 
           6348354750 \gamma^6 + 4226605272 \gamma^7 + 1965585132 \gamma^8 + 
           619358832 \gamma^9 + 124502562 \gamma^{10} + 14161824 \gamma^{11} + 
           676476 \gamma^{12}) + 
        2 \beta^4 (6931801 + 88562133 \gamma + 501425205 \gamma^2 + 
           1667076219 \gamma^3 + 3634966569 \gamma^4 + 5492132661 \gamma^5 + 
           5912513424 \gamma^6 + 4582465242 \gamma^7 + 2547008730 \gamma^8 + 
           996508260 \gamma^9 + 264297492 \gamma^{10} + 44381898 \gamma^{11} + 
           4139928 \gamma^{12} + 156276 \gamma^{13}) + 
        2 \beta^3 (2504924 + 34531804 \gamma + 212141602 \gamma^2 + 770439968 \gamma^3 + 
           1849820919 \gamma^4 + 3107393420 \gamma^5 + 3763080608 \gamma^6 + 
           3329142856 \gamma^7 + 2152325077 \gamma^8 + 1004762976 \gamma^9 + 
           329890686 \gamma^{10} + 72722772 \gamma^{11} + 9933378 \gamma^{12} + 
           722808 \gamma^{13} + 19512 \gamma^{14}) + 
        2 \beta^2 (620542 + 9181200 \gamma + 60841566 \gamma^2 + 239767342 \gamma^3 + 
           629015253 \gamma^4 + 1163883711 \gamma^5 + 1567365714 \gamma^6 + 
           1559700420 \gamma^7 + 1150345833 \gamma^8 + 623767833 \gamma^9 + 
           243699084 \gamma^{10} + 66183462 \gamma^{11} + 11772198 \gamma^{12} + 
           1238562 \gamma^{13} + 63288 \gamma^{14} + 936 \gamma^{15}) + 
        \beta (189455 + 2994294 \gamma + 21295602 \gamma^2 + 90557264 \gamma^3 + 
           257928066 \gamma^4 + 521755656 \gamma^5 + 774295554 \gamma^6 + 
           857016816 \gamma^7 + 710858991 \gamma^8 + 439357458 \gamma^9 + 
           198973668 \gamma^{10} + 64031904 \gamma^{11} + 13920840 \gamma^{12} + 
           1881288 \gamma^{13} + 136224 \gamma^{14} + 3744 \gamma^{15})) + 
     32 \alpha^{10} (101571 + 1813896 \gamma + 14604432 \gamma^2 + 70513188 \gamma^3 + 
        228947238 \gamma^4 + 530714796 \gamma^5 + 908577086 \gamma^6 + 
        1170005160 \gamma^7 + 1141232271 \gamma^8 + 840744108 \gamma^9 + 
        461642418 \gamma^{10} + 184071084 \gamma^{11} + 50987928 \gamma^{12} + 
        9116424 \gamma^{13} + 926928 \gamma^{14} + 41184 \gamma^{15} + 
        9504 \beta^{16} (1 + \gamma)^2 + 
        144 \beta^{15} (1183 + 3483 \gamma + 3318 \gamma^2 + 1024 \gamma^3) + 
        216 \beta^{14} (6316 + 24618 \gamma + 34789 \gamma^2 + 21128 \gamma^3 + 
           4660 \gamma^4) + 
        72 \beta^{13} (96823 + 471661 \gamma + 888960 \gamma^2 + 810587 \gamma^3 + 
           358098 \gamma^4 + 61446 \gamma^5) + 
        36 \beta^{12} (735754 + 4301928 \gamma + 10141263 \gamma^2 + 12349542 \gamma^3 + 
           8201203 \gamma^4 + 2818212 \gamma^5 + 391604 \gamma^6) + 
        36 \beta^{11} (2177829 + 14844821 \gamma + 41930424 \gamma^2 + 63707682 \gamma^3 + 
           56273970 \gamma^4 + 28901574 \gamma^5 + 7986696 \gamma^6 + 914952 \gamma^7) + 
        12 \beta^{10} (15092248 + 117478328 \gamma + 386258192 \gamma^2 + 
           701643432 \gamma^3 + 770739321 \gamma^4 + 524286162 \gamma^5 + 
           215528052 \gamma^6 + 48883464 \gamma^7 + 4672512 \gamma^8) + 
        6 \beta^9 (54113941 + 473877053 \gamma + 1777613374 \gamma^2 + 
           3753671772 \gamma^3 + 4920957942 \gamma^4 + 4154238888 \gamma^5 + 
           2257053528 \gamma^6 + 759927912 \gamma^7 + 143506800 \gamma^8 + 
           11535408 \gamma^9) + 
        9 \beta^8 (50073183 + 487683172 \gamma + 2056551614 \gamma^2 + 
           4949644268 \gamma^3 + 7535259120 \gamma^4 + 7584600552 \gamma^5 + 
           5109646056 \gamma^6 + 2271692592 \gamma^7 + 636119736 \gamma^8 + 
           100862688 \gamma^9 + 6831552 \gamma^{10}) + 
        12 \beta^7 (40310411 + 432617955 \gamma + 2027291830 \gamma^2 + 
           5479651292 \gamma^3 + 9500210432 \gamma^4 + 11097897116 \gamma^5 + 
           8910332832 \gamma^6 + 4909496100 \gamma^7 + 1813823784 \gamma^8 + 
           425847144 \gamma^9 + 56757960 \gamma^{10} + 3217536 \gamma^{11}) + 
        2 \beta^6 (200060251 + 2347598010 \gamma + 12110199180 \gamma^2 + 
           36337115952 \gamma^3 + 70695226782 \gamma^4 + 93998991036 \gamma^5 + 
           87553159008 \gamma^6 + 57463109352 \gamma^7 + 26287614180 \gamma^8 + 
           8130225024 \gamma^9 + 1600273764 \gamma^{10} + 177741216 \gamma^{11} + 
           8262576 \gamma^{12}) + 
        36 \beta^5 (7022006 + 89492871 \gamma + 504203554 \gamma^2 + 
           1663729339 \gamma^3 + 3590601953 \gamma^4 + 5354821229 \gamma^5 + 
           5674613530 \gamma^6 + 4318168326 \gamma^7 + 2350739340 \gamma^8 + 
           898687554 \gamma^9 + 232349904 \gamma^{10} + 37934550 \gamma^{11} + 
           3429204 \gamma^{12} + 124908 \gamma^{13}) + 
        6 \beta^4 (19982221 + 274972842 \gamma + 1680693627 \gamma^2 + 
           6051628508 \gamma^3 + 14354192721 \gamma^4 + 23737332056 \gamma^5 + 
           28203704840 \gamma^6 + 24404002612 \gamma^7 + 15387164296 \gamma^8 + 
           6987360408 \gamma^9 + 2226322194 \gamma^{10} + 475151076 \gamma^{11} + 
           62660778 \gamma^{12} + 4384440 \gamma^{13} + 113112 \gamma^{14}) + 
        4 \beta^3 (10351008 + 153003966 \gamma + 1008770556 \gamma^2 + 
           3937971914 \gamma^3 + 10188924021 \gamma^4 + 18514557807 \gamma^5 + 
           24387085926 \gamma^6 + 23647686948 \gamma^7 + 16937706327 \gamma^8 + 
           8892593871 \gamma^9 + 3355419528 \gamma^{10} + 878283522 \gamma^{11} + 
           150278058 \gamma^{12} + 15165558 \gamma^{13} + 738504 \gamma^{14} + 
           10296 \gamma^{15}) + 
        6 \beta (243015 + 4090221 \gamma + 30934242 \gamma^2 + 139747562 \gamma^3 + 
           422702658 \gamma^4 + 908461190 \gamma^5 + 1434347920 \gamma^6 + 
           1693404494 \gamma^7 + 1504292741 \gamma^8 + 1001688285 \gamma^9 + 
           492947758 \gamma^{10} + 174516976 \gamma^{11} + 42492654 \gamma^{12} + 
           6606096 \gamma^{13} + 574344 \gamma^{14} + 20592 \gamma^{15}) + 
        6 \beta^2 (1644079 + 25985194 \gamma + 183889668 \gamma^2 + 774046148 \gamma^3 + 
           2171082743 \gamma^4 + 4303446630 \gamma^5 + 6228336672 \gamma^6 + 
           6693270000 \gamma^7 + 5368240158 \gamma^8 + 3196556436 \gamma^9 + 
           1390549224 \gamma^{10} + 428995368 \gamma^{11} + 89342802 \gamma^{12} + 
           11564244 \gamma^{13} + 799128 \gamma^{14} + 20592 \gamma^{15})) + 
     32 \alpha^9 (98661 + 1866687 \gamma + 15885618 \gamma^2 + 80906116 \gamma^3 + 
        276675866 \gamma^4 + 674840994 \gamma^5 + 1215258340 \gamma^6 + 
        1646801188 \gamma^7 + 1692453549 \gamma^8 + 1316646147 \gamma^9 + 
        766098846 \gamma^{10} + 325359792 \gamma^{11} + 96708696 \gamma^{12} + 
        18750876 \gamma^{13} + 2095344 \gamma^{14} + 102960 \gamma^{15} + 
        31680 \beta^{16} (1 + \gamma)^3 + 
        144 \beta^{15} (3577 + 14088 \gamma + 20307 \gamma^2 + 12718 \gamma^3 + 
           2919 \gamma^4) + 
        144 \beta^{14} (26455 + 129298 \gamma + 244923 \gamma^2 + 224942 \gamma^3 + 
           100275 \gamma^4 + 17373 \gamma^5) + 
        36 \beta^{13} (494521 + 2894334 \gamma + 6824448 \gamma^2 + 8301192 \gamma^3 + 
           5500993 \gamma^4 + 1885908 \gamma^5 + 261668 \gamma^6) + 
        36 \beta^{12} (1679936 + 11461538 \gamma + 32375393 \gamma^2 + 49113434 \gamma^3 + 
           43259667 \gamma^4 + 22147391 \gamma^5 + 6106160 \gamma^6 + 699228 \gamma^7) + 
        6 \beta^{11} (26373871 + 205456130 \gamma + 675804032 \gamma^2 + 
           1227179856 \gamma^3 + 1346713032 \gamma^4 + 915029208 \gamma^5 + 
           375865032 \gamma^6 + 85255200 \gamma^7 + 8158608 \gamma^8) + 
        6 \beta^{10} (54113941 + 473877053 \gamma + 1777613374 \gamma^2 + 
           3753671772 \gamma^3 + 4920957942 \gamma^4 + 4154238888 \gamma^5 + 
           2257053528 \gamma^6 + 759927912 \gamma^7 + 143506800 \gamma^8 + 
           11535408 \gamma^9) + 
        3 \beta^9 (174607269 + 1698703196 \gamma + 7156258100 \gamma^2 + 
           17212146352 \gamma^3 + 26193016172 \gamma^4 + 26354408496 \gamma^5 + 
           17743004784 \gamma^6 + 7879002048 \gamma^7 + 2202068472 \gamma^8 + 
           348202080 \gamma^9 + 23498784 \gamma^{10}) + 
        3 \beta^8 (221197459 + 2369005197 \gamma + 11078121098 \gamma^2 + 
           29888655160 \gamma^3 + 51730358632 \gamma^4 + 60315385804 \gamma^5 + 
           48307759368 \gamma^6 + 26528528208 \gamma^7 + 9757218456 \gamma^8 + 
           2277491496 \gamma^9 + 301345920 \gamma^{10} + 16932384 \gamma^{11}) + 
        4 \beta^7 (164524979 + 1925435724 \gamma + 9902508978 \gamma^2 + 
           29619619992 \gamma^3 + 57428407107 \gamma^4 + 76051747446 \gamma^5 + 
           70487991102 \gamma^6 + 45981130104 \gamma^7 + 20877441012 \gamma^8 + 
           6398299368 \gamma^9 + 1245700656 \gamma^{10} + 136586952 \gamma^{11} + 
           6254460 \gamma^{12}) + 
        4 \beta^6 (126990284 + 1613821209 \gamma + 9058957899 \gamma^2 + 
           29760798636 \gamma^3 + 63889144377 \gamma^4 + 94669904847 \gamma^5 + 
           99548350620 \gamma^6 + 75056010414 \gamma^7 + 40418483490 \gamma^8 + 
           15258697452 \gamma^9 + 3888266652 \gamma^{10} + 624347964 \gamma^{11} + 
           55373292 \gamma^{12} + 1973268 \gamma^{13}) + 
        6 \beta^5 (50317235 + 690580122 \gamma + 4203806094 \gamma^2 + 
           15052823248 \gamma^3 + 35449700907 \gamma^4 + 58102621684 \gamma^5 + 
           68298924532 \gamma^6 + 58360769120 \gamma^7 + 36273044390 \gamma^8 + 
           16207508316 \gamma^9 + 5071711008 \gamma^{10} + 1060891200 \gamma^{11} + 
           136791678 \gamma^{12} + 9330264 \gamma^{13} + 233784 \gamma^{14}) + 
        2 \beta^4 (67819149 + 1000352193 \gamma + 6567233148 \gamma^2 + 
           25468663750 \gamma^3 + 65307839103 \gamma^4 + 117329699037 \gamma^5 + 
           152438589738 \gamma^6 + 145478832156 \gamma^7 + 102339946632 \gamma^8 + 
           52671163188 \gamma^9 + 19448043642 \gamma^{10} + 4972632300 \gamma^{11} + 
           829469862 \gamma^{12} + 81379638 \gamma^{13} + 3836160 \gamma^{14} + 
           51480 \gamma^{15}) + 
        6 \beta^2 (1701427 + 28625021 \gamma + 215574820 \gamma^2 + 965871822 \gamma^3 + 
           2885988454 \gamma^4 + 6103704054 \gamma^5 + 9449758694 \gamma^6 + 
           10903990516 \gamma^7 + 9439482997 \gamma^8 + 6110444049 \gamma^9 + 
           2917972298 \gamma^{10} + 1001551064 \gamma^{11} + 236491542 \gamma^{12} + 
           35688288 \gamma^{13} + 3007080 \gamma^{14} + 102960 \gamma^{15}) + 
        2 \beta^3 (22328959 + 352438856 \gamma + 2483323794 \gamma^2 + 
           10375060656 \gamma^3 + 28791287962 \gamma^4 + 56287911900 \gamma^5 + 
           80115084192 \gamma^6 + 84441054384 \gamma^7 + 66262973733 \gamma^8 + 
           38524767414 \gamma^9 + 16335200220 \gamma^{10} + 4905694944 \gamma^{11} + 
           993416184 \gamma^{12} + 124822296 \gamma^{13} + 8339040 \gamma^{14} + 
           205920 \gamma^{15}) + 
        \beta (1457295 + 26041596 \gamma + 208835076 \gamma^2 + 999439596 \gamma^3 + 
           3201285246 \gamma^4 + 7287607092 \gamma^5 + 12200387516 \gamma^6 + 
           15302837928 \gamma^7 + 14486663067 \gamma^8 + 10325801124 \gamma^9 + 
           5473297116 \gamma^{10} + 2105117664 \gamma^{11} + 563502180 \gamma^{12} + 
           97900848 \gamma^{13} + 9729360 \gamma^{14} + 411840 \gamma^{15})) + 
     16 \alpha^8 (150265 + 3004570 \gamma + 26965276 \gamma^2 + 144567028 \gamma^3 + 
        519661332 \gamma^4 + 1331111864 \gamma^5 + 2516570692 \gamma^6 + 
        3581417380 \gamma^7 + 3869506165 \gamma^8 + 3170504622 \gamma^9 + 
        1948355292 \gamma^{10} + 877378944 \gamma^{11} + 278036964 \gamma^{12} + 
        57887208 \gamma^{13} + 6997104 \gamma^{14} + 370656 \gamma^{15} + 
        142560 \beta^{16} (1 + \gamma)^4 + 
        2592 \beta^{15} (830 + 4095 \gamma + 7915 \gamma^2 + 7505 \gamma^3 + 3492 \gamma^4 + 
           637 \gamma^5) + 
        72 \beta^{14} (207007 + 1217220 \gamma + 2895576 \gamma^2 + 3571200 \gamma^3 + 
           2410947 \gamma^4 + 845100 \gamma^5 + 120132 \gamma^6) + 
        72 \beta^{13} (901681 + 6168525 \gamma + 17491740 \gamma^2 + 26668446 \gamma^3 + 
           23638455 \gamma^4 + 12194991 \gamma^5 + 3392712 \gamma^6 + 392652 \gamma^7) + 
        36 \beta^{12} (5635478 + 43998642 \gamma + 145053360 \gamma^2 + 
           263893830 \gamma^3 + 290089857 \gamma^4 + 197508510 \gamma^5 + 
           81380580 \gamma^6 + 18547128 \gamma^7 + 1787400 \gamma^8) + 
        36 \beta^{11} (13427795 + 117807109 \gamma + 442670586 \gamma^2 + 
           935756704 \gamma^3 + 1227490494 \gamma^4 + 1036889712 \gamma^5 + 
           564058200 \gamma^6 + 190371336 \gamma^7 + 36093600 \gamma^8 + 
           2918064 \gamma^9) + 
        18 \beta^{10} (50073183 + 487683172 \gamma + 2056551614 \gamma^2 + 
           4949644268 \gamma^3 + 7535259120 \gamma^4 + 7584600552 \gamma^5 + 
           5109646056 \gamma^6 + 2271692592 \gamma^7 + 636119736 \gamma^8 + 
           100862688 \gamma^9 + 6831552 \gamma^{10}) + 
        6 \beta^9 (221197459 + 2369005197 \gamma + 11078121098 \gamma^2 + 
           29888655160 \gamma^3 + 51730358632 \gamma^4 + 60315385804 \gamma^5 + 
           48307759368 \gamma^6 + 26528528208 \gamma^7 + 9757218456 \gamma^8 + 
           2277491496 \gamma^9 + 301345920 \gamma^{10} + 16932384 \gamma^{11}) + 
        3 \beta^8 (515286439 + 6023638698 \gamma + 30945901236 \gamma^2 + 
           92478350736 \gamma^3 + 179156226348 \gamma^4 + 237044796120 \gamma^5 + 
           219456132624 \gamma^6 + 142938970560 \gamma^7 + 64767673080 \gamma^8 + 
           19796405328 \gamma^9 + 3841323840 \gamma^{10} + 419477184 \gamma^{11} + 
           19115712 \gamma^{12}) + 
        24 \beta^7 (59196732 + 750783614 \gamma + 4205782614 \gamma^2 + 
           13790115876 \gamma^3 + 29545581783 \gamma^4 + 43682271435 \gamma^5 + 
           45807672930 \gamma^6 + 34418728254 \gamma^7 + 18455017626 \gamma^8 + 
           6930030060 \gamma^9 + 1754545572 \gamma^{10} + 279570456 \gamma^{11} + 
           24572160 \gamma^{12} + 866556 \gamma^{13}) + 
        4 \beta^6 (255947350 + 3504003826 \gamma + 21270659626 \gamma^2 + 
           75937680414 \gamma^3 + 178236009003 \gamma^4 + 290981558700 \gamma^5 + 
           340427239320 \gamma^6 + 289237167468 \gamma^7 + 178550842014 \gamma^8 + 
           79141751424 \gamma^9 + 24533825652 \gamma^{10} + 5076324216 \gamma^{11} + 
           646364070 \gamma^{12} + 43455096 \gamma^{13} + 1071144 \gamma^{14}) + 
        12 \beta^5 (47659660 + 701176360 \gamma + 4588005685 \gamma^2 + 
           17721801346 \gamma^3 + 45222255813 \gamma^4 + 80767168443 \gamma^5 + 
           104199735414 \gamma^6 + 98625849300 \gamma^7 + 68724178938 \gamma^8 + 
           34989967698 \gamma^9 + 12762985068 \gamma^{10} + 3218911596 \gamma^{11} + 
           528701562 \gamma^{12} + 50968386 \gamma^{13} + 2354832 \gamma^{14} + 
           30888 \gamma^{15}) + 
        6 \beta (381843 + 7224301 \gamma + 61268624 \gamma^2 + 309843880 \gamma^3 + 
           1048326156 \gamma^4 + 2521280140 \gamma^5 + 4463129238 \gamma^6 + 
           5928682604 \gamma^7 + 5958496239 \gamma^8 + 4524519393 \gamma^9 + 
           2566897266 \gamma^{10} + 1063266960 \gamma^{11} + 309024024 \gamma^{12} + 
           58895988 \gamma^{13} + 6498576 \gamma^{14} + 308880 \gamma^{15}) + 
        6 \beta^4 (40503030 + 637811820 \gamma + 4477913206 \gamma^2 + 
           18615418430 \gamma^3 + 51326159139 \gamma^4 + 99544920786 \gamma^5 + 
           140339133144 \gamma^6 + 146297320176 \gamma^7 + 113387065128 \gamma^8 + 
           65023425228 \gamma^9 + 27160977024 \gamma^{10} + 8025325164 \gamma^{11} + 
           1596648078 \gamma^{12} + 196718724 \gamma^{13} + 12848760 \gamma^{14} + 
           308880 \gamma^{15}) + 
        6 \beta^2 (2776867 + 49573058 \gamma + 396079080 \gamma^2 + 
           1883108180 \gamma^3 + 5974814036 \gamma^4 + 13436181548 \gamma^5 + 
           22164972632 \gamma^6 + 27333749652 \gamma^7 + 25392155721 \gamma^8 + 
           17734012500 \gamma^9 + 9201881118 \gamma^{10} + 3463987236 \gamma^{11} + 
           908231880 \gamma^{12} + 154752984 \gamma^{13} + 15070320 \gamma^{14} + 
           617760 \gamma^{15}) + 
        4 \beta^3 (19033200 + 319649947 \gamma + 2398301942 \gamma^2 + 
           10682700308 \gamma^3 + 31664022014 \gamma^4 + 66291010880 \gamma^5 + 
           101393402664 \gamma^6 + 115376640372 \gamma^7 + 98340272271 \gamma^8 + 
           62592326025 \gamma^9 + 29358364290 \gamma^{10} + 9889812348 \gamma^{11} + 
           2290463622 \gamma^{12} + 338684976 \gamma^{13} + 27886680 \gamma^{14} + 
           926640 \gamma^{15})) + 
     2 \alpha (83 + 2301 \gamma + 28494 \gamma^2 + 209856 \gamma^3 + 1033632 \gamma^4 + 
        3628668 \gamma^5 + 9431904 \gamma^6 + 18562128 \gamma^7 + 27957616 \gamma^8 + 
        32247792 \gamma^9 + 28207200 \gamma^{10} + 18287232 \gamma^{11} + 8424576 \gamma^{12} + 
        2555712 \gamma^{13} + 440064 \gamma^{14} + 29952 \gamma^{15} + 
        27648 \beta^{16} (1 + \gamma)^{11} + 
        2304 \beta^{15} (1 + \gamma)^8 (127 + 496 \gamma + 627 \gamma^2 + 334 \gamma^3 + 
           49 \gamma^4) + 
        2304 \beta^{14} (1 + \gamma)^5 (647 + 5071 \gamma + 16370 \gamma^2 + 28608 \gamma^3 + 
           29724 \gamma^4 + 18603 \gamma^5 + 6654 \gamma^6 + 1170 \gamma^7 + 75 \gamma^8) + 
        576 \beta^{13} (1 + \gamma)^2 (8331 + 98324 \gamma + 511329 \gamma^2 + 
           1552182 \gamma^3 + 3067556 \gamma^4 + 4155814 \gamma^5 + 3943268 \gamma^6 + 
           2620230 \gamma^7 + 1194293 \gamma^8 + 356638 \gamma^9 + 64183 \gamma^{10} + 
           5988 \gamma^{11} + 204 \gamma^{12}) + 
        192 \beta^{12} (56504 + 835890 \gamma + 5572947 \gamma^2 + 22237910 \gamma^3 + 
           59449941 \gamma^4 + 112792509 \gamma^5 + 156685238 \gamma^6 + 
           161771940 \gamma^7 + 124481856 \gamma^8 + 70789394 \gamma^9 + 
           29137383 \gamma^{10} + 8366670 \gamma^{11} + 1577247 \gamma^{12} + 
           176295 \gamma^{13} + 9648 \gamma^{14} + 156 \gamma^{15}) + 
        96 \beta^{11} (189455 + 2994294 \gamma + 21295602 \gamma^2 + 90557264 \gamma^3 + 
           257928066 \gamma^4 + 521755656 \gamma^5 + 774295554 \gamma^6 + 
           857016816 \gamma^7 + 710858991 \gamma^8 + 439357458 \gamma^9 + 
           198973668 \gamma^{10} + 64031904 \gamma^{11} + 13920840 \gamma^{12} + 
           1881288 \gamma^{13} + 136224 \gamma^{14} + 3744 \gamma^{15}) + 
        96 \beta^{10} (243015 + 4090221 \gamma + 30934242 \gamma^2 + 139747562 \gamma^3 + 
           422702658 \gamma^4 + 908461190 \gamma^5 + 1434347920 \gamma^6 + 
           1693404494 \gamma^7 + 1504292741 \gamma^8 + 1001688285 \gamma^9 + 
           492947758 \gamma^{10} + 174516976 \gamma^{11} + 42492654 \gamma^{12} + 
           6606096 \gamma^{13} + 574344 \gamma^{14} + 20592 \gamma^{15}) + 
        48 \beta^8 (381843 + 7224301 \gamma + 61268624 \gamma^2 + 309843880 \gamma^3 + 
           1048326156 \gamma^4 + 2521280140 \gamma^5 + 4463129238 \gamma^6 + 
           5928682604 \gamma^7 + 5958496239 \gamma^8 + 4524519393 \gamma^9 + 
           2566897266 \gamma^{10} + 1063266960 \gamma^{11} + 309024024 \gamma^{12} + 
           58895988 \gamma^{13} + 6498576 \gamma^{14} + 308880 \gamma^{15}) + 
        6 \beta^2 (3185 + 81009 \gamma + 921162 \gamma^2 + 6234348 \gamma^3 + 
           28226558 \gamma^4 + 91061128 \gamma^5 + 217316288 \gamma^6 + 
           392138896 \gamma^7 + 540659952 \gamma^8 + 569912368 \gamma^9 + 
           454923904 \gamma^{10} + 268979136 \gamma^{11} + 113128992 \gamma^{12} + 
           31481472 \gamma^{13} + 5037696 \gamma^{14} + 329472 \gamma^{15}) + 
        \beta (1809 + 48060 \gamma + 570708 \gamma^2 + 4033104 \gamma^3 + 19068228 \gamma^4 + 
           64260720 \gamma^5 + 160302576 \gamma^6 + 302616960 \gamma^7 + 
           436926448 \gamma^8 + 482788800 \gamma^9 + 404323392 \gamma^{10} + 
           250933248 \gamma^{11} + 110738880 \gamma^{12} + 32262912 \gamma^{13} + 
           5370624 \gamma^{14} + 359424 \gamma^{15}) + 
        16 \beta^9 (1457295 + 26041596 \gamma + 208835076 \gamma^2 + 
           999439596 \gamma^3 + 3201285246 \gamma^4 + 7287607092 \gamma^5 + 
           12200387516 \gamma^6 + 15302837928 \gamma^7 + 14486663067 \gamma^8 + 
           10325801124 \gamma^9 + 5473297116 \gamma^{10} + 2105117664 \gamma^{11} + 
           563502180 \gamma^{12} + 97900848 \gamma^{13} + 9729360 \gamma^{14} + 
           411840 \gamma^{15}) + 
        32 \beta^7 (354935 + 7091552 \gamma + 63454583 \gamma^2 + 338329577 \gamma^3 + 
           1206486897 \gamma^4 + 3058821745 \gamma^5 + 5712295973 \gamma^6 + 
           8016255824 \gamma^7 + 8528788544 \gamma^8 + 6875034108 \gamma^9 + 
           4155639972 \gamma^{10} + 1842405000 \gamma^{11} + 576332784 \gamma^{12} + 
           118976976 \gamma^{13} + 14302656 \gamma^{14} + 741312 \gamma^{15}) + 
        16 \beta^6 (345385 + 7270481 \gamma + 68491211 \gamma^2 + 384249297 \gamma^3 + 
           1441436474 \gamma^4 + 3845306644 \gamma^5 + 7561803840 \gamma^6 + 
           11188995262 \gamma^7 + 12575001478 \gamma^8 + 10733008212 \gamma^9 + 
           6889139244 \gamma^{10} + 3254378856 \gamma^{11} + 1088800752 \gamma^{12} + 
           241266240 \gamma^{13} + 31187808 \gamma^{14} + 1729728 \gamma^{15}) + 
        12 \beta^5 (174031 + 3851418 \gamma + 38124084 \gamma^2 + 224651400 \gamma^3 + 
           885078867 \gamma^4 + 2480613356 \gamma^5 + 5129258140 \gamma^6 + 
           7990589104 \gamma^7 + 9470862808 \gamma^8 + 8542536640 \gamma^9 + 
           5808053376 \gamma^{10} + 2913566592 \gamma^{11} + 1037630064 \gamma^{12} + 
           245146944 \gamma^{13} + 33753024 \gamma^{14} + 1976832 \gamma^{15}) + 
        2 \beta^3 (63997 + 1556454 \gamma + 16924368 \gamma^2 + 109526656 \gamma^3 + 
           474047592 \gamma^4 + 1461121912 \gamma^5 + 3328610784 \gamma^6 + 
           5727354368 \gamma^7 + 7520422096 \gamma^8 + 7539996000 \gamma^9 + 
           5717867904 \gamma^{10} + 3209132928 \gamma^{11} + 1281185280 \gamma^{12} + 
           339018624 \gamma^{13} + 51881472 \gamma^{14} + 3294720 \gamma^{15}) + 
        4 \beta^4 (150360 + 3491430 \gamma + 36251457 \gamma^2 + 224020938 \gamma^3 + 
           925651323 \gamma^4 + 2722187171 \gamma^5 + 5911492860 \gamma^6 + 
           9684086620 \gamma^7 + 12089051624 \gamma^8 + 11505163368 \gamma^9 + 
           8269245504 \gamma^{10} + 4393188672 \gamma^{11} + 1659250800 \gamma^{12} + 
           415777392 \gamma^{13} + 60531840 \gamma^{14} + 3706560 \gamma^{15})) + 
     32 \alpha^7 (44965 + 947906 \gamma + 8953208 \gamma^2 + 50433780 \gamma^3 + 
        190237205 \gamma^4 + 510946738 \gamma^5 + 1012667298 \gamma^6 + 
        1511386522 \gamma^7 + 1714215562 \gamma^8 + 1476815916 \gamma^9 + 
        956433924 \gamma^{10} + 455317224 \gamma^{11} + 153148944 \gamma^{12} + 
        34002144 \gamma^{13} + 4396896 \gamma^{14} + 247104 \gamma^{15} + 
        114048 \beta^{16} (1 + \gamma)^5 + 
        864 \beta^{15} (1 + \gamma)^2 (1873 + 7360 \gamma + 10347 \gamma^2 + 6226 \gamma^3 + 
           1348 \gamma^4) + 
        576 \beta^{14} (18401 + 126494 \gamma + 362490 \gamma^2 + 562161 \gamma^3 + 
           510108 \gamma^4 + 270909 \gamma^5 + 77916 \gamma^6 + 9348 \gamma^7) + 
        144 \beta^{13} (302167 + 2367099 \gamma + 7850734 \gamma^2 + 14411244 \gamma^3 + 
           16033422 \gamma^4 + 11080326 \gamma^5 + 4645605 \gamma^6 + 1079640 \gamma^7 + 
           106311 \gamma^8) + 
        24 \beta^{12} (5295623 + 46597014 \gamma + 175791546 \gamma^2 + 
           373427518 \gamma^3 + 492789420 \gamma^4 + 419333004 \gamma^5 + 
           230155380 \gamma^6 + 78517062 \gamma^7 + 15079320 \gamma^8 + 
           1237998 \gamma^9) + 
        36 \beta^{11} (7801383 + 76183516 \gamma + 322221824 \gamma^2 + 
           777806072 \gamma^3 + 1187803396 \gamma^4 + 1199952552 \gamma^5 + 
           812143884 \gamma^6 + 363251736 \gamma^7 + 102511560 \gamma^8 + 
           16414440 \gamma^9 + 1125264 \gamma^{10}) + 
        12 \beta^{10} (40310411 + 432617955 \gamma + 2027291830 \gamma^2 + 
           5479651292 \gamma^3 + 9500210432 \gamma^4 + 11097897116 \gamma^5 + 
           8910332832 \gamma^6 + 4909496100 \gamma^7 + 1813823784 \gamma^8 + 
           425847144 \gamma^9 + 56757960 \gamma^{10} + 3217536 \gamma^{11}) + 
        4 \beta^9 (164524979 + 1925435724 \gamma + 9902508978 \gamma^2 + 
           29619619992 \gamma^3 + 57428407107 \gamma^4 + 76051747446 \gamma^5 + 
           70487991102 \gamma^6 + 45981130104 \gamma^7 + 20877441012 \gamma^8 + 
           6398299368 \gamma^9 + 1245700656 \gamma^{10} + 136586952 \gamma^{11} + 
           6254460 \gamma^{12}) + 
        12 \beta^8 (59196732 + 750783614 \gamma + 4205782614 \gamma^2 + 
           13790115876 \gamma^3 + 29545581783 \gamma^4 + 43682271435 \gamma^5 + 
           45807672930 \gamma^6 + 34418728254 \gamma^7 + 18455017626 \gamma^8 + 
           6930030060 \gamma^9 + 1754545572 \gamma^{10} + 279570456 \gamma^{11} + 
           24572160 \gamma^{12} + 866556 \gamma^{13}) + 
        2 \beta^7 (303843389 + 4155395864 \gamma + 25199241314 \gamma^2 + 
           89879646204 \gamma^3 + 210767779518 \gamma^4 + 343744321032 \gamma^5 + 
           401664366600 \gamma^6 + 340744294368 \gamma^7 + 209944274796 \gamma^8 + 
           92836334016 \gamma^9 + 28696258008 \gamma^{10} + 5917165776 \gamma^{11} + 
           750385512 \gamma^{12} + 50213088 \gamma^{13} + 1231200 \gamma^{14}) + 
        2 \beta^6 (204806499 + 3007685993 \gamma + 19643603096 \gamma^2 + 
           75735965954 \gamma^3 + 192883467858 \gamma^4 + 343720583982 \gamma^5 + 
           442263814836 \gamma^6 + 417268888848 \gamma^7 + 289647444492 \gamma^8 + 
           146799004692 \gamma^9 + 53259240312 \gamma^{10} + 13347995544 \gamma^{11} + 
           2176399368 \gamma^{12} + 208048392 \gamma^{13} + 9520416 \gamma^{14} + 
           123552 \gamma^{15}) + 
        6 \beta^5 (35852801 + 563320483 \gamma + 3945044496 \gamma^2 + 
           16354965660 \gamma^3 + 44951419536 \gamma^4 + 86859260112 \gamma^5 + 
           121923443826 \gamma^6 + 126457725168 \gamma^7 + 97440288978 \gamma^8 + 
           55507282956 \gamma^9 + 23010905880 \gamma^{10} + 6740792352 \gamma^{11} + 
           1327997232 \gamma^{12} + 161791920 \gamma^{13} + 10431936 \gamma^{14} + 
           247104 \gamma^{15}) + 
        6 \beta^2 (895593 + 16918230 \gamma + 143019770 \gamma^2 + 719571964 \gamma^3 + 
           2417517258 \gamma^4 + 5763228960 \gamma^5 + 10096664438 \gamma^6 + 
           13256112870 \gamma^7 + 13153915698 \gamma^8 + 9854448738 \gamma^9 + 
           5514133236 \gamma^{10} + 2253373488 \gamma^{11} + 646681608 \gamma^{12} + 
           121835640 \gamma^{13} + 13279680 \gamma^{14} + 617760 \gamma^{15}) + 
        2 \beta (354935 + 7091552 \gamma + 63454583 \gamma^2 + 338329577 \gamma^3 + 
           1206486897 \gamma^4 + 3058821745 \gamma^5 + 5712295973 \gamma^6 + 
           8016255824 \gamma^7 + 8528788544 \gamma^8 + 6875034108 \gamma^9 + 
           4155639972 \gamma^{10} + 1842405000 \gamma^{11} + 576332784 \gamma^{12} + 
           118976976 \gamma^{13} + 14302656 \gamma^{14} + 741312 \gamma^{15}) + 
        2 \beta^3 (12852393 + 228969683 \gamma + 1823576014 \gamma^2 + 
           8631201200 \gamma^3 + 27226445866 \gamma^4 + 60792203020 \gamma^5 + 
           99454958954 \gamma^6 + 121503100404 \gamma^7 + 111718158762 \gamma^8 + 
           77169318624 \gamma^9 + 39580415520 \gamma^{10} + 14722040208 \gamma^{11} + 
           3812580552 \gamma^{12} + 641201184 \gamma^{13} + 61512480 \gamma^{14} + 
           2471040 \gamma^{15}) + 
        \beta^4 (86373735 + 1447238734 \gamma + 10826508380 \gamma^2 + 
           48048252116 \gamma^3 + 141780998438 \gamma^4 + 295241111048 \gamma^5 + 
           448758270192 \gamma^6 + 507016522116 \gamma^7 + 428719232916 \gamma^8 + 
           270487522788 \gamma^9 + 125658027888 \gamma^{10} + 41889305184 \gamma^{11} + 
           9590687064 \gamma^{12} + 1400023008 \gamma^{13} + 113568480 \gamma^{14} + 
           3706560 \gamma^{15})) + 
     2 \alpha^2 (945 + 25140 \gamma + 298662 \gamma^2 + 2109468 \gamma^3 + 9958968 \gamma^4 + 
        33486216 \gamma^5 + 83285952 \gamma^6 + 156665280 \gamma^7 + 225273008 \gamma^8 + 
        247789632 \gamma^9 + 206497248 \gamma^{10} + 127489728 \gamma^{11} + 
        55960704 \gamma^{12} + 16219008 \gamma^{13} + 2688768 \gamma^{14} + 179712 \gamma^{15} + 
        152064 \beta^{16} (1 + \gamma)^{10} + 
        6912 \beta^{15} (1 + \gamma)^7 (245 + 958 \gamma + 1239 \gamma^2 + 674 \gamma^3 + 
           112 \gamma^4) + 
        1152 \beta^{14} (1 + \gamma)^4 (7766 + 60934 \gamma + 198419 \gamma^2 + 
           352104 \gamma^3 + 373782 \gamma^4 + 241794 \gamma^5 + 91599 \gamma^6 + 
           18024 \gamma^7 + 1380 \gamma^8) + 
        1152 \beta^{13} (25747 + 329837 \gamma + 1894068 \gamma^2 + 6464745 \gamma^3 + 
           14648610 \gamma^4 + 23275746 \gamma^5 + 26654972 \gamma^6 + 
           22217956 \gamma^7 + 13419003 \gamma^8 + 5759625 \gamma^9 + 1690284 \gamma^{10} + 
           316587 \gamma^{11} + 33246 \gamma^{12} + 1434 \gamma^{13}) + 
        576 \beta^{12} (119760 + 1652202 \gamma + 10213099 \gamma^2 + 37535814 \gamma^3 + 
           91728185 \gamma^4 + 157708514 \gamma^5 + 196517596 \gamma^6 + 
           179793816 \gamma^7 + 120771194 \gamma^8 + 58824688 \gamma^9 + 
           20224939 \gamma^{10} + 4683438 \gamma^{11} + 674153 \gamma^{12} + 51948 \gamma^{13} + 
           1500 \gamma^{14}) + 
        192 \beta^{11} (620542 + 9181200 \gamma + 60841566 \gamma^2 + 239767342 \gamma^3 + 
           629015253 \gamma^4 + 1163883711 \gamma^5 + 1567365714 \gamma^6 + 
           1559700420 \gamma^7 + 1150345833 \gamma^8 + 623767833 \gamma^9 + 
           243699084 \gamma^{10} + 66183462 \gamma^{11} + 11772198 \gamma^{12} + 
           1238562 \gamma^{13} + 63288 \gamma^{14} + 936 \gamma^{15}) + 
        96 \beta^{10} (1644079 + 25985194 \gamma + 183889668 \gamma^2 + 
           774046148 \gamma^3 + 2171082743 \gamma^4 + 4303446630 \gamma^5 + 
           6228336672 \gamma^6 + 6693270000 \gamma^7 + 5368240158 \gamma^8 + 
           3196556436 \gamma^9 + 1390549224 \gamma^{10} + 428995368 \gamma^{11} + 
           89342802 \gamma^{12} + 11564244 \gamma^{13} + 799128 \gamma^{14} + 
           20592 \gamma^{15}) + 
        96 \beta^9 (1701427 + 28625021 \gamma + 215574820 \gamma^2 + 
           965871822 \gamma^3 + 2885988454 \gamma^4 + 6103704054 \gamma^5 + 
           9449758694 \gamma^6 + 10903990516 \gamma^7 + 9439482997 \gamma^8 + 
           6110444049 \gamma^9 + 2917972298 \gamma^{10} + 1001551064 \gamma^{11} + 
           236491542 \gamma^{12} + 35688288 \gamma^{13} + 3007080 \gamma^{14} + 
           102960 \gamma^{15}) + 
        6 \beta (3185 + 81009 \gamma + 921162 \gamma^2 + 6234348 \gamma^3 + 
           28226558 \gamma^4 + 91061128 \gamma^5 + 217316288 \gamma^6 + 
           392138896 \gamma^7 + 540659952 \gamma^8 + 569912368 \gamma^9 + 
           454923904 \gamma^{10} + 268979136 \gamma^{11} + 113128992 \gamma^{12} + 
           31481472 \gamma^{13} + 5037696 \gamma^{14} + 329472 \gamma^{15}) + 
        96 \beta^7 (895593 + 16918230 \gamma + 143019770 \gamma^2 + 719571964 \gamma^3 + 
           2417517258 \gamma^4 + 5763228960 \gamma^5 + 10096664438 \gamma^6 + 
           13256112870 \gamma^7 + 13153915698 \gamma^8 + 9854448738 \gamma^9 + 
           5514133236 \gamma^{10} + 2253373488 \gamma^{11} + 646681608 \gamma^{12} + 
           121835640 \gamma^{13} + 13279680 \gamma^{14} + 617760 \gamma^{15}) + 
        48 \beta^8 (2776867 + 49573058 \gamma + 396079080 \gamma^2 + 
           1883108180 \gamma^3 + 5974814036 \gamma^4 + 13436181548 \gamma^5 + 
           22164972632 \gamma^6 + 27333749652 \gamma^7 + 25392155721 \gamma^8 + 
           17734012500 \gamma^9 + 9201881118 \gamma^{10} + 3463987236 \gamma^{11} + 
           908231880 \gamma^{12} + 154752984 \gamma^{13} + 15070320 \gamma^{14} + 
           617760 \gamma^{15}) + 
        12 \beta^2 (15676 + 380808 \gamma + 4138164 \gamma^2 + 26777400 \gamma^3 + 
           115942673 \gamma^4 + 357676818 \gamma^5 + 815921628 \gamma^6 + 
           1406382712 \gamma^7 + 1850635512 \gamma^8 + 1860042160 \gamma^9 + 
           1414406304 \gamma^{10} + 796140480 \gamma^{11} + 318774480 \gamma^{12} + 
           84578976 \gamma^{13} + 12968640 \gamma^{14} + 823680 \gamma^{15}) + 
        24 \beta^5 (720623 + 15135703 \gamma + 142248126 \gamma^2 + 795917647 \gamma^3 + 
           2976842776 \gamma^4 + 7915916984 \gamma^5 + 15515636436 \gamma^6 + 
           22883885252 \gamma^7 + 25639598540 \gamma^8 + 21822604520 \gamma^9 + 
           13973350440 \gamma^{10} + 6588305760 \gamma^{11} + 2201223264 \gamma^{12} + 
           487245888 \gamma^{13} + 62852544 \gamma^{14} + 3459456 \gamma^{15}) + 
        12 \beta^4 (436982 + 9649116 \gamma + 95335075 \gamma^2 + 560869562 \gamma^3 + 
           2206678385 \gamma^4 + 6178052652 \gamma^5 + 12765779996 \gamma^6 + 
           19882120952 \gamma^7 + 23570605216 \gamma^8 + 21275405616 \gamma^9 + 
           14482429936 \gamma^{10} + 7276899456 \gamma^{11} + 2596600464 \gamma^{12} + 
           614604864 \gamma^{13} + 84680640 \gamma^{14} + 4942080 \gamma^{15}) + 
        4 \beta^3 (295437 + 6847089 \gamma + 71001480 \gamma^2 + 438448794 \gamma^3 + 
           1811366850 \gamma^4 + 5329074502 \gamma^5 + 11583967920 \gamma^6 + 
           19006172792 \gamma^7 + 23776227568 \gamma^8 + 22687038528 \gamma^9 + 
           16355870112 \gamma^{10} + 8718601632 \gamma^{11} + 3304336608 \gamma^{12} + 
           830651616 \gamma^{13} + 121184640 \gamma^{14} + 7413120 \gamma^{15}) + 
        8 \beta^6 (5459257 + 108860824 \gamma + 971350369 \gamma^2 + 
           5159281168 \gamma^3 + 18308495454 \gamma^4 + 46149064052 \gamma^5 + 
           85619908768 \gamma^6 + 119302759096 \gamma^7 + 125986723924 \gamma^8 + 
           100788060504 \gamma^9 + 60467961072 \gamma^{10} + 26621173632 \gamma^{11} + 
           8275782672 \gamma^{12} + 1699061472 \gamma^{13} + 202958784 \gamma^{14} + 
           10378368 \gamma^{15})) + 
     8 \alpha^5 (30513 + 710253 \gamma + 7388664 \gamma^2 + 45729222 \gamma^3 + 
        189176247 \gamma^4 + 556770619 \gamma^5 + 1209448620 \gamma^6 + 
        1980822668 \gamma^7 + 2470651048 \gamma^8 + 2347662648 \gamma^9 + 
        1683260592 \gamma^{10} + 891089664 \gamma^{11} + 334901232 \gamma^{12} + 
        83403504 \gamma^{13} + 12075264 \gamma^{14} + 741312 \gamma^{15} + 
        456192 \beta^{16} (1 + \gamma)^7 + 
        5184 \beta^{15} (1 + \gamma)^4 (1127 + 4420 \gamma + 6036 \gamma^2 + 3488 \gamma^3 + 
           703 \gamma^4) + 
        576 \beta^{14} (60407 + 535473 \gamma + 2057886 \gamma^2 + 4508000 \gamma^3 + 
           6210534 \gamma^4 + 5581812 \gamma^5 + 3270186 \gamma^6 + 1201788 \gamma^7 + 
           250533 \gamma^8 + 22473 \gamma^9) + 
        144 \beta^{13} (899019 + 8832892 \gamma + 37842492 \gamma^2 + 93216404 \gamma^3 + 
           146380586 \gamma^4 + 153205140 \gamma^5 + 108182904 \gamma^6 + 
           50806152 \gamma^7 + 15141591 \gamma^8 + 2574636 \gamma^9 + 188532 \gamma^{10}) + 
        144 \beta^{12} (2357095 + 25444653 \gamma + 120398730 \gamma^2 + 
           329916916 \gamma^3 + 582373152 \gamma^4 + 695822880 \gamma^5 + 
           574052682 \gamma^6 + 326509164 \gamma^7 + 125098407 \gamma^8 + 
           30603689 \gamma^9 + 4272360 \gamma^{10} + 255228 \gamma^{11}) + 
        48 \beta^{11} (13878853 + 163335480 \gamma + 846436347 \gamma^2 + 
           2555731368 \gamma^3 + 5012833029 \gamma^4 + 6733073982 \gamma^5 + 
           6348354750 \gamma^6 + 4226605272 \gamma^7 + 1965585132 \gamma^8 + 
           619358832 \gamma^9 + 124502562 \gamma^{10} + 14161824 \gamma^{11} + 
           676476 \gamma^{12}) + 
        144 \beta^{10} (7022006 + 89492871 \gamma + 504203554 \gamma^2 + 
           1663729339 \gamma^3 + 3590601953 \gamma^4 + 5354821229 \gamma^5 + 
           5674613530 \gamma^6 + 4318168326 \gamma^7 + 2350739340 \gamma^8 + 
           898687554 \gamma^9 + 232349904 \gamma^{10} + 37934550 \gamma^{11} + 
           3429204 \gamma^{12} + 124908 \gamma^{13}) + 
        24 \beta^9 (50317235 + 690580122 \gamma + 4203806094 \gamma^2 + 
           15052823248 \gamma^3 + 35449700907 \gamma^4 + 58102621684 \gamma^5 + 
           68298924532 \gamma^6 + 58360769120 \gamma^7 + 36273044390 \gamma^8 + 
           16207508316 \gamma^9 + 5071711008 \gamma^{10} + 1060891200 \gamma^{11} + 
           136791678 \gamma^{12} + 9330264 \gamma^{13} + 233784 \gamma^{14}) + 
        24 \beta^8 (47659660 + 701176360 \gamma + 4588005685 \gamma^2 + 
           17721801346 \gamma^3 + 45222255813 \gamma^4 + 80767168443 \gamma^5 + 
           104199735414 \gamma^6 + 98625849300 \gamma^7 + 68724178938 \gamma^8 + 
           34989967698 \gamma^9 + 12762985068 \gamma^{10} + 3218911596 \gamma^{11} + 
           528701562 \gamma^{12} + 50968386 \gamma^{13} + 2354832 \gamma^{14} + 
           30888 \gamma^{15}) + 
        24 \beta^7 (35852801 + 563320483 \gamma + 3945044496 \gamma^2 + 
           16354965660 \gamma^3 + 44951419536 \gamma^4 + 86859260112 \gamma^5 + 
           121923443826 \gamma^6 + 126457725168 \gamma^7 + 97440288978 \gamma^8 + 
           55507282956 \gamma^9 + 23010905880 \gamma^{10} + 6740792352 \gamma^{11} + 
           1327997232 \gamma^{12} + 161791920 \gamma^{13} + 10431936 \gamma^{14} + 
           247104 \gamma^{15}) + 
        3 \beta (174031 + 3851418 \gamma + 38124084 \gamma^2 + 224651400 \gamma^3 + 
           885078867 \gamma^4 + 2480613356 \gamma^5 + 5129258140 \gamma^6 + 
           7990589104 \gamma^7 + 9470862808 \gamma^8 + 8542536640 \gamma^9 + 
           5808053376 \gamma^{10} + 2913566592 \gamma^{11} + 1037630064 \gamma^{12} + 
           245146944 \gamma^{13} + 33753024 \gamma^{14} + 1976832 \gamma^{15}) + 
        8 \beta^6 (63993216 + 1070490502 \gamma + 7994749364 \gamma^2 + 
           35419909934 \gamma^3 + 104322724331 \gamma^4 + 216783436778 \gamma^5 + 
           328715624856 \gamma^6 + 370374850422 \gamma^7 + 312205548330 \gamma^8 + 
           196280852046 \gamma^9 + 90816883080 \gamma^{10} + 30134466744 \gamma^{11} + 
           6862414068 \gamma^{12} + 995515920 \gamma^{13} + 80169264 \gamma^{14} + 
           2594592 \gamma^{15}) + 
        6 \beta^2 (720623 + 15135703 \gamma + 142248126 \gamma^2 + 795917647 \gamma^3 + 
           2976842776 \gamma^4 + 7915916984 \gamma^5 + 15515636436 \gamma^6 + 
           22883885252 \gamma^7 + 25639598540 \gamma^8 + 21822604520 \gamma^9 + 
           13973350440 \gamma^{10} + 6588305760 \gamma^{11} + 2201223264 \gamma^{12} + 
           487245888 \gamma^{13} + 62852544 \gamma^{14} + 3459456 \gamma^{15}) + 
        3 \beta^5 (79488585 + 1412003268 \gamma + 11210791228 \gamma^2 + 
           52884030240 \gamma^3 + 166192972712 \gamma^4 + 369513115712 \gamma^5 + 
           601659842592 \gamma^6 + 731206082112 \gamma^7 + 668480236008 \gamma^8 + 
           458870633952 \gamma^9 + 233741574048 \gamma^{10} + 86276603904 \gamma^{11} + 
           22149541152 \gamma^{12} + 3687707520 \gamma^{13} + 349550208 \gamma^{14} + 
           13837824 \gamma^{15}) + 
        3 \beta^4 (28473753 + 535885939 \gamma + 4511109796 \gamma^2 + 
           22586245900 \gamma^3 + 75454653016 \gamma^4 + 178732102592 \gamma^5 + 
           310923229400 \gamma^6 + 405136603424 \gamma^7 + 398809282680 \gamma^8 + 
           296283261960 \gamma^9 + 164345241744 \gamma^{10} + 66546195264 \gamma^{11} + 
           18910101504 \gamma^{12} + 3523378656 \gamma^{13} + 378870912 \gamma^{14} + 
           17297280 \gamma^{15}) + 
        2 \beta^3 (11418400 + 227226352 \gamma + 2023268693 \gamma^2 + 
           10722383632 \gamma^3 + 37957401440 \gamma^4 + 95429590116 \gamma^5 + 
           176578805276 \gamma^6 + 245389944608 \gamma^7 + 258460187792 \gamma^8 + 
           206237762832 \gamma^9 + 123422213280 \gamma^{10} + 54198878880 \gamma^{11} + 
           16802792640 \gamma^{12} + 3438437184 \gamma^{13} + 408893184 \gamma^{14} + 
           20756736 \gamma^{15})) + 
     4 \alpha^3 (3375 + 86033 \gamma + 979230 \gamma^2 + 6625272 \gamma^3 + 
        29950718 \gamma^4 + 96368312 \gamma^5 + 229144480 \gamma^6 + 411607600 \gamma^7 + 
        564473648 \gamma^8 + 591405360 \gamma^9 + 468905088 \gamma^{10} + 
        275222400 \gamma^{11} + 114864288 \gamma^{12} + 31720320 \gamma^{13} + 
        5044608 \gamma^{14} + 329472 \gamma^{15} + 253440 \beta^{16} (1 + \gamma)^9 + 
        1152 \beta^{15} (1 + \gamma)^6 (2569 + 10056 \gamma + 13269 \gamma^2 + 7366 \gamma^3 + 
           1332 \gamma^4) + 
        2304 \beta^{14} (1 + \gamma)^3 (7049 + 55358 \gamma + 181741 \gamma^2 + 
           327194 \gamma^3 + 354381 \gamma^4 + 235938 \gamma^5 + 93493 \gamma^6 + 
           19848 \gamma^7 + 1704 \gamma^8) + 
        192 \beta^{13} (290296 + 3430374 \gamma + 18029571 \gamma^2 + 55812228 \gamma^3 + 
           113467326 \gamma^4 + 159661986 \gamma^5 + 159313772 \gamma^6 + 
           113319432 \gamma^7 + 56794014 \gamma^8 + 19438740 \gamma^9 + 
           4275501 \gamma^{10} + 536256 \gamma^{11} + 28530 \gamma^{12}) + 
        192 \beta^{12} (699045 + 8945188 \gamma + 50961936 \gamma^2 + 171338739 \gamma^3 + 
           379706143 \gamma^4 + 585951258 \gamma^5 + 647262722 \gamma^6 + 
           516970566 \gamma^7 + 297272430 \gamma^8 + 120747642 \gamma^9 + 
           33353010 \gamma^{10} + 5851569 \gamma^{11} + 572508 \gamma^{12} + 22806 \gamma^{13}) + 
        96 \beta^{11} (2504924 + 34531804 \gamma + 212141602 \gamma^2 + 
           770439968 \gamma^3 + 1849820919 \gamma^4 + 3107393420 \gamma^5 + 
           3763080608 \gamma^6 + 3329142856 \gamma^7 + 2152325077 \gamma^8 + 
           1004762976 \gamma^9 + 329890686 \gamma^{10} + 72722772 \gamma^{11} + 
           9933378 \gamma^{12} + 722808 \gamma^{13} + 19512 \gamma^{14}) + 
        32 \beta^{10} (10351008 + 153003966 \gamma + 1008770556 \gamma^2 + 
           3937971914 \gamma^3 + 10188924021 \gamma^4 + 18514557807 \gamma^5 + 
           24387085926 \gamma^6 + 23647686948 \gamma^7 + 16937706327 \gamma^8 + 
           8892593871 \gamma^9 + 3355419528 \gamma^{10} + 878283522 \gamma^{11} + 
           150278058 \gamma^{12} + 15165558 \gamma^{13} + 738504 \gamma^{14} + 
           10296 \gamma^{15}) + 
        16 \beta^9 (22328959 + 352438856 \gamma + 2483323794 \gamma^2 + 
           10375060656 \gamma^3 + 28791287962 \gamma^4 + 56287911900 \gamma^5 + 
           80115084192 \gamma^6 + 84441054384 \gamma^7 + 66262973733 \gamma^8 + 
           38524767414 \gamma^9 + 16335200220 \gamma^{10} + 4905694944 \gamma^{11} + 
           993416184 \gamma^{12} + 124822296 \gamma^{13} + 8339040 \gamma^{14} + 
           205920 \gamma^{15}) + 
        16 \beta^8 (19033200 + 319649947 \gamma + 2398301942 \gamma^2 + 
           10682700308 \gamma^3 + 31664022014 \gamma^4 + 66291010880 \gamma^5 + 
           101393402664 \gamma^6 + 115376640372 \gamma^7 + 98340272271 \gamma^8 + 
           62592326025 \gamma^9 + 29358364290 \gamma^{10} + 9889812348 \gamma^{11} + 
           2290463622 \gamma^{12} + 338684976 \gamma^{13} + 27886680 \gamma^{14} + 
           926640 \gamma^{15}) + 
        16 \beta^7 (12852393 + 228969683 \gamma + 1823576014 \gamma^2 + 
           8631201200 \gamma^3 + 27226445866 \gamma^4 + 60792203020 \gamma^5 + 
           99454958954 \gamma^6 + 121503100404 \gamma^7 + 111718158762 \gamma^8 + 
           77169318624 \gamma^9 + 39580415520 \gamma^{10} + 14722040208 \gamma^{11} + 
           3812580552 \gamma^{12} + 641201184 \gamma^{13} + 61512480 \gamma^{14} + 
           2471040 \gamma^{15}) + 
        \beta (63997 + 1556454 \gamma + 16924368 \gamma^2 + 109526656 \gamma^3 + 
           474047592 \gamma^4 + 1461121912 \gamma^5 + 3328610784 \gamma^6 + 
           5727354368 \gamma^7 + 7520422096 \gamma^8 + 7539996000 \gamma^9 + 
           5717867904 \gamma^{10} + 3209132928 \gamma^{11} + 1281185280 \gamma^{12} + 
           339018624 \gamma^{13} + 51881472 \gamma^{14} + 3294720 \gamma^{15}) + 
        2 \beta^2 (295437 + 6847089 \gamma + 71001480 \gamma^2 + 438448794 \gamma^3 + 
           1811366850 \gamma^4 + 5329074502 \gamma^5 + 11583967920 \gamma^6 + 
           19006172792 \gamma^7 + 23776227568 \gamma^8 + 22687038528 \gamma^9 + 
           16355870112 \gamma^{10} + 8718601632 \gamma^{11} + 3304336608 \gamma^{12} + 
           830651616 \gamma^{13} + 121184640 \gamma^{14} + 7413120 \gamma^{15}) + 
        8 \beta^6 (13701587 + 258270755 \gamma + 2177502142 \gamma^2 + 
           10919613352 \gamma^3 + 36540698068 \gamma^4 + 86709780580 \gamma^5 + 
           151124772832 \gamma^6 + 197303902068 \gamma^7 + 194618759220 \gamma^8 + 
           144896066124 \gamma^9 + 80558079816 \gamma^{10} + 32703773280 \gamma^{11} + 
           9321403056 \gamma^{12} + 1743177744 \gamma^{13} + 188310528 \gamma^{14} + 
           8648640 \gamma^{15}) + 
        4 \beta^4 (3654702 + 76635274 \gamma + 719183610 \gamma^2 + 
           4018718301 \gamma^3 + 15012476870 \gamma^4 + 39878634556 \gamma^5 + 
           78097733600 \gamma^6 + 115114949380 \gamma^7 + 128929971196 \gamma^8 + 
           109720249080 \gamma^9 + 70256502600 \gamma^{10} + 33126966144 \gamma^{11} + 
           11066904672 \gamma^{12} + 2448324864 \gamma^{13} + 315351360 \gamma^{14} + 
           17297280 \gamma^{15}) + 
        2 \beta^3 (1745337 + 38501672 \gamma + 380120866 \gamma^2 + 
           2235103532 \gamma^3 + 8790798894 \gamma^4 + 24608666168 \gamma^5 + 
           50854858808 \gamma^6 + 79232030848 \gamma^7 + 93985633216 \gamma^8 + 
           84899915904 \gamma^9 + 57845665056 \gamma^{10} + 29093607360 \gamma^{11} + 
           10390616160 \gamma^{12} + 2460924288 \gamma^{13} + 339085440 \gamma^{14} + 
           19768320 \gamma^{15}) + 
        4 \beta^5 (11418400 + 227226352 \gamma + 2023268693 \gamma^2 + 
           10722383632 \gamma^3 + 37957401440 \gamma^4 + 95429590116 \gamma^5 + 
           176578805276 \gamma^6 + 245389944608 \gamma^7 + 258460187792 \gamma^8 + 
           206237762832 \gamma^9 + 123422213280 \gamma^{10} + 54198878880 \gamma^{11} + 
           16802792640 \gamma^{12} + 3438437184 \gamma^{13} + 408893184 \gamma^{14} + 
           20756736 \gamma^{15})) + 
     8 \alpha^6 (84231 + 1867880 \gamma + 18531878 \gamma^2 + 109501792 \gamma^3 + 
        432809387 \gamma^4 + 1217400396 \gamma^5 + 2526813284 \gamma^6 + 
        3951435944 \gamma^7 + 4700685728 \gamma^8 + 4253940816 \gamma^9 + 
        2899691520 \gamma^{10} + 1456519776 \gamma^{11} + 518405040 \gamma^{12} + 
        122122944 \gamma^{13} + 16759872 \gamma^{14} + 988416 \gamma^{15} + 
        532224 \beta^{16} (1 + \gamma)^6 + 
        3456 \beta^{15} (1 + \gamma)^3 (2071 + 8130 \gamma + 11271 \gamma^2 + 6646 \gamma^3 + 
           1392 \gamma^4) + 
        576 \beta^{14} (77431 + 609326 \gamma + 2043643 \gamma^2 + 3821832 \gamma^3 + 
           4363764 \gamma^4 + 3115962 \gamma^5 + 1357749 \gamma^6 + 329448 \gamma^7 + 
           33984 \gamma^8) + 
        192 \beta^{13} (905437 + 7993941 \gamma + 30374742 \gamma^2 + 65266730 \gamma^3 + 
           87502725 \gamma^4 + 75961359 \gamma^5 + 42688926 \gamma^6 + 
           14957271 \gamma^7 + 2958003 \gamma^8 + 250683 \gamma^9) + 
        48 \beta^{12} (9987579 + 97834318 \gamma + 415912116 \gamma^2 + 
           1011241464 \gamma^3 + 1559036120 \gamma^4 + 1593783936 \gamma^5 + 
           1094130354 \gamma^6 + 497515428 \gamma^7 + 143064153 \gamma^8 + 
           23399484 \gamma^9 + 1643148 \gamma^{10}) + 
        48 \beta^{11} (20712775 + 222961725 \gamma + 1048949606 \gamma^2 + 
           2848718020 \gamma^3 + 4967195728 \gamma^4 + 5843028064 \gamma^5 + 
           4731024660 \gamma^6 + 2633298420 \gamma^7 + 984662748 \gamma^8 + 
           234477204 \gamma^9 + 31773744 \gamma^{10} + 1836288 \gamma^{11}) + 
        8 \beta^{10} (200060251 + 2347598010 \gamma + 12110199180 \gamma^2 + 
           36337115952 \gamma^3 + 70695226782 \gamma^4 + 93998991036 \gamma^5 + 
           87553159008 \gamma^6 + 57463109352 \gamma^7 + 26287614180 \gamma^8 + 
           8130225024 \gamma^9 + 1600273764 \gamma^{10} + 177741216 \gamma^{11} + 
           8262576 \gamma^{12}) + 
        16 \beta^9 (126990284 + 1613821209 \gamma + 9058957899 \gamma^2 + 
           29760798636 \gamma^3 + 63889144377 \gamma^4 + 94669904847 \gamma^5 + 
           99548350620 \gamma^6 + 75056010414 \gamma^7 + 40418483490 \gamma^8 + 
           15258697452 \gamma^9 + 3888266652 \gamma^{10} + 624347964 \gamma^{11} + 
           55373292 \gamma^{12} + 1973268 \gamma^{13}) + 
        8 \beta^8 (255947350 + 3504003826 \gamma + 21270659626 \gamma^2 + 
           75937680414 \gamma^3 + 178236009003 \gamma^4 + 290981558700 \gamma^5 + 
           340427239320 \gamma^6 + 289237167468 \gamma^7 + 178550842014 \gamma^8 + 
           79141751424 \gamma^9 + 24533825652 \gamma^{10} + 5076324216 \gamma^{11} + 
           646364070 \gamma^{12} + 43455096 \gamma^{13} + 1071144 \gamma^{14}) + 
        8 \beta^7 (204806499 + 3007685993 \gamma + 19643603096 \gamma^2 + 
           75735965954 \gamma^3 + 192883467858 \gamma^4 + 343720583982 \gamma^5 + 
           442263814836 \gamma^6 + 417268888848 \gamma^7 + 289647444492 \gamma^8 + 
           146799004692 \gamma^9 + 53259240312 \gamma^{10} + 13347995544 \gamma^{11} + 
           2176399368 \gamma^{12} + 208048392 \gamma^{13} + 9520416 \gamma^{14} + 
           123552 \gamma^{15}) + 
        8 \beta^6 (129486260 + 2032597228 \gamma + 14221719442 \gamma^2 + 
           58907561024 \gamma^3 + 161761558443 \gamma^4 + 312258611670 \gamma^5 + 
           437804542428 \gamma^6 + 453461699616 \gamma^7 + 348839177136 \gamma^8 + 
           198334427376 \gamma^9 + 82033536528 \gamma^{10} + 23966377776 \gamma^{11} + 
           4706760564 \gamma^{12} + 571337928 \gamma^{13} + 36684144 \gamma^{14} + 
           864864 \gamma^{15}) + 
        4 \beta (345385 + 7270481 \gamma + 68491211 \gamma^2 + 384249297 \gamma^3 + 
           1441436474 \gamma^4 + 3845306644 \gamma^5 + 7561803840 \gamma^6 + 
           11188995262 \gamma^7 + 12575001478 \gamma^8 + 10733008212 \gamma^9 + 
           6889139244 \gamma^{10} + 3254378856 \gamma^{11} + 1088800752 \gamma^{12} + 
           241266240 \gamma^{13} + 31187808 \gamma^{14} + 1729728 \gamma^{15}) + 
        8 \beta^5 (63993216 + 1070490502 \gamma + 7994749364 \gamma^2 + 
           35419909934 \gamma^3 + 104322724331 \gamma^4 + 216783436778 \gamma^5 + 
           328715624856 \gamma^6 + 370374850422 \gamma^7 + 312205548330 \gamma^8 + 
           196280852046 \gamma^9 + 90816883080 \gamma^{10} + 30134466744 \gamma^{11} + 
           6862414068 \gamma^{12} + 995515920 \gamma^{13} + 80169264 \gamma^{14} + 
           2594592 \gamma^{15}) + 
        4 \beta^3 (13701587 + 258270755 \gamma + 2177502142 \gamma^2 + 
           10919613352 \gamma^3 + 36540698068 \gamma^4 + 86709780580 \gamma^5 + 
           151124772832 \gamma^6 + 197303902068 \gamma^7 + 194618759220 \gamma^8 + 
           144896066124 \gamma^9 + 80558079816 \gamma^{10} + 32703773280 \gamma^{11} + 
           9321403056 \gamma^{12} + 1743177744 \gamma^{13} + 188310528 \gamma^{14} + 
           8648640 \gamma^{15}) + 
        2 \beta^2 (5459257 + 108860824 \gamma + 971350369 \gamma^2 + 
           5159281168 \gamma^3 + 18308495454 \gamma^4 + 46149064052 \gamma^5 + 
           85619908768 \gamma^6 + 119302759096 \gamma^7 + 125986723924 \gamma^8 + 
           100788060504 \gamma^9 + 60467961072 \gamma^{10} + 26621173632 \gamma^{11} + 
           8275782672 \gamma^{12} + 1699061472 \gamma^{13} + 202958784 \gamma^{14} + 
           10378368 \gamma^{15}) + 
        \beta^4 (194157927 + 3451903660 \gamma + 27429618440 \gamma^2 + 
           129497609080 \gamma^3 + 407301400616 \gamma^4 + 906414504560 \gamma^5 + 
           1477331108872 \gamma^6 + 1797369375120 \gamma^7 + 
           1645144802952 \gamma^8 + 1130793039072 \gamma^9 + 
           576879747024 \gamma^{10} + 213307601568 \gamma^{11} + 54876213312 \gamma^{12} + 
           9159272640 \gamma^{13} + 870791040 \gamma^{14} + 34594560 \gamma^{15})) + 
     4 \alpha^4 (16756 + 408558 \gamma + 4449084 \gamma^2 + 28805710 \gamma^3 + 
        124612737 \gamma^4 + 383532430 \gamma^5 + 871687260 \gamma^6 + 
        1495054424 \gamma^7 + 1955180872 \gamma^8 + 1950734256 \gamma^9 + 
        1470877824 \gamma^{10} + 820102368 \gamma^{11} + 324992016 \gamma^{12} + 
        85324896 \gamma^{13} + 12972096 \gamma^{14} + 823680 \gamma^{15} + 
        570240 \beta^{16} (1 + \gamma)^8 + 
        1152 \beta^{15} (1 + \gamma)^5 (6052 + 23713 \gamma + 31857 \gamma^2 + 18043 \gamma^3 + 
           3469 \gamma^4) + 
        288 \beta^{14} (1 + \gamma)^2 (138235 + 1086422 \gamma + 3594073 \gamma^2 + 
           6558600 \gamma^3 + 7237977 \gamma^4 + 4943598 \gamma^5 + 2032059 \gamma^6 + 
           455916 \gamma^7 + 42372 \gamma^8) + 
        288 \beta^{13} (493035 + 5335293 \gamma + 25448048 \gamma^2 + 70719196 \gamma^3 + 
           127376618 \gamma^4 + 156208130 \gamma^5 + 133005792 \gamma^6 + 
           78471876 \gamma^7 + 31326087 \gamma^8 + 8015953 \gamma^9 + 1174824 \gamma^{10} + 
           73980 \gamma^{11}) + 
        48 \beta^{12} (7414629 + 87459870 \gamma + 456004260 \gamma^2 + 
           1391022300 \gamma^3 + 2768146284 \gamma^4 + 3787978404 \gamma^5 + 
           3652951880 \gamma^6 + 2496465108 \gamma^7 + 1195608141 \gamma^8 + 
           389129022 \gamma^9 + 81031560 \gamma^{10} + 9579528 \gamma^{11} + 
           477648 \gamma^{12}) + 
        96 \beta^{11} (6931801 + 88562133 \gamma + 501425205 \gamma^2 + 
           1667076219 \gamma^3 + 3634966569 \gamma^4 + 5492132661 \gamma^5 + 
           5912513424 \gamma^6 + 4582465242 \gamma^7 + 2547008730 \gamma^8 + 
           996508260 \gamma^9 + 264297492 \gamma^{10} + 44381898 \gamma^{11} + 
           4139928 \gamma^{12} + 156276 \gamma^{13}) + 
        48 \beta^{10} (19982221 + 274972842 \gamma + 1680693627 \gamma^2 + 
           6051628508 \gamma^3 + 14354192721 \gamma^4 + 23737332056 \gamma^5 + 
           28203704840 \gamma^6 + 24404002612 \gamma^7 + 15387164296 \gamma^8 + 
           6987360408 \gamma^9 + 2226322194 \gamma^{10} + 475151076 \gamma^{11} + 
           62660778 \gamma^{12} + 4384440 \gamma^{13} + 113112 \gamma^{14}) + 
        16 \beta^9 (67819149 + 1000352193 \gamma + 6567233148 \gamma^2 + 
           25468663750 \gamma^3 + 65307839103 \gamma^4 + 117329699037 \gamma^5 + 
           152438589738 \gamma^6 + 145478832156 \gamma^7 + 102339946632 \gamma^8 + 
           52671163188 \gamma^9 + 19448043642 \gamma^{10} + 4972632300 \gamma^{11} + 
           829469862 \gamma^{12} + 81379638 \gamma^{13} + 3836160 \gamma^{14} + 
           51480 \gamma^{15}) + 
        24 \beta^8 (40503030 + 637811820 \gamma + 4477913206 \gamma^2 + 
           18615418430 \gamma^3 + 51326159139 \gamma^4 + 99544920786 \gamma^5 + 
           140339133144 \gamma^6 + 146297320176 \gamma^7 + 113387065128 \gamma^8 + 
           65023425228 \gamma^9 + 27160977024 \gamma^{10} + 8025325164 \gamma^{11} + 
           1596648078 \gamma^{12} + 196718724 \gamma^{13} + 12848760 \gamma^{14} + 
           308880 \gamma^{15}) + 
        2 \beta (150360 + 3491430 \gamma + 36251457 \gamma^2 + 224020938 \gamma^3 + 
           925651323 \gamma^4 + 2722187171 \gamma^5 + 5911492860 \gamma^6 + 
           9684086620 \gamma^7 + 12089051624 \gamma^8 + 11505163368 \gamma^9 + 
           8269245504 \gamma^{10} + 4393188672 \gamma^{11} + 1659250800 \gamma^{12} + 
           415777392 \gamma^{13} + 60531840 \gamma^{14} + 3706560 \gamma^{15}) + 
        8 \beta^7 (86373735 + 1447238734 \gamma + 10826508380 \gamma^2 + 
           48048252116 \gamma^3 + 141780998438 \gamma^4 + 295241111048 \gamma^5 + 
           448758270192 \gamma^6 + 507016522116 \gamma^7 + 428719232916 \gamma^8 + 
           270487522788 \gamma^9 + 125658027888 \gamma^{10} + 41889305184 \gamma^{11} + 
           9590687064 \gamma^{12} + 1400023008 \gamma^{13} + 113568480 \gamma^{14} + 
           3706560 \gamma^{15}) + 
        6 \beta^2 (436982 + 9649116 \gamma + 95335075 \gamma^2 + 560869562 \gamma^3 + 
           2206678385 \gamma^4 + 6178052652 \gamma^5 + 12765779996 \gamma^6 + 
           19882120952 \gamma^7 + 23570605216 \gamma^8 + 21275405616 \gamma^9 + 
           14482429936 \gamma^{10} + 7276899456 \gamma^{11} + 2596600464 \gamma^{12} + 
           614604864 \gamma^{13} + 84680640 \gamma^{14} + 4942080 \gamma^{15}) + 
        4 \beta^3 (3654702 + 76635274 \gamma + 719183610 \gamma^2 + 4018718301 \gamma^3 +
            15012476870 \gamma^4 + 39878634556 \gamma^5 + 78097733600 \gamma^6 + 
           115114949380 \gamma^7 + 128929971196 \gamma^8 + 109720249080 \gamma^9 + 
           70256502600 \gamma^{10} + 33126966144 \gamma^{11} + 11066904672 \gamma^{12} + 
           2448324864 \gamma^{13} + 315351360 \gamma^{14} + 17297280 \gamma^{15}) + 
        6 \beta^5 (28473753 + 535885939 \gamma + 4511109796 \gamma^2 + 
           22586245900 \gamma^3 + 75454653016 \gamma^4 + 178732102592 \gamma^5 + 
           310923229400 \gamma^6 + 405136603424 \gamma^7 + 398809282680 \gamma^8 + 
           296283261960 \gamma^9 + 164345241744 \gamma^{10} + 66546195264 \gamma^{11} + 
           18910101504 \gamma^{12} + 3523378656 \gamma^{13} + 378870912 \gamma^{14} + 
           17297280 \gamma^{15}) + 
        2 \beta^6 (194157927 + 3451903660 \gamma + 27429618440 \gamma^2 + 
           129497609080 \gamma^3 + 407301400616 \gamma^4 + 906414504560 \gamma^5 + 
           1477331108872 \gamma^6 + 1797369375120 \gamma^7 + 
           1645144802952 \gamma^8 + 1130793039072 \gamma^9 + 
           576879747024 \gamma^{10} + 213307601568 \gamma^{11} + 54876213312 \gamma^{12} + 
           9159272640 \gamma^{13} + 870791040 \gamma^{14} + 34594560 \gamma^{15}) + 
        \beta^4 (57822431 + 1149672506 \gamma + 10228852192 \gamma^2 + 
           54168317768 \gamma^3 + 191621892808 \gamma^4 + 481439052720 \gamma^5 + 
           890281694656 \gamma^6 + 1236523230496 \gamma^7 + 1301730372712 \gamma^8 + 
           1038228358128 \gamma^9 + 621026192544 \gamma^{10} + 
           272558371200 \gamma^{11} + 84433956192 \gamma^{12} + 17259185088 \gamma^{13} + 
           2049183360 \gamma^{14} + 103783680 \gamma^{15}))\Big]\Big/\\
           \Big[1 + 10 \gamma + 36 \gamma^2 + 
   64 \gamma^3 + 60 \gamma^4 + 24 \gamma^5 + 24 \beta^5 (1 + \gamma)^5 + 
   24 \alpha^5 (1 + \beta + \gamma)^5 + 
   12 \beta^4 (1 + \gamma)^2 (5 + 20 \gamma + 23 \gamma^2 + 10 \gamma^3) + 
   16 \beta^3 (4 + 28 \gamma + 72 \gamma^2 + 90 \gamma^3 + 57 \gamma^4 + 15 \gamma^5) + 
   12 \beta^2 (3 + 24 \gamma + 69 \gamma^2 + 96 \gamma^3 + 68 \gamma^4 + 20 \gamma^5) + 
   2 \beta (5 + 45 \gamma + 144 \gamma^2 + 224 \gamma^3 + 180 \gamma^4 + 60 \gamma^5) + 
   12 \alpha^4 (1 + \beta + \gamma)^2 (5 + 20 \gamma + 23 \gamma^2 + 10 \gamma^3 + 10 \beta^3 (1 + \gamma) +
       \beta^2 (23 + 46 \gamma + 16 \gamma^2) + 2 \beta (10 + 30 \gamma + 23 \gamma^2 + 5 \gamma^3)) + 
   16 \alpha^3 (4 + 28 \gamma + 72 \gamma^2 + 90 \gamma^3 + 57 \gamma^4 + 15 \gamma^5 + 
      15 \beta^5 (1 + \gamma)^2 + 3 \beta^4 (19 + 57 \gamma + 50 \gamma^2 + 13 \gamma^3) + 
      3 \beta^3 (30 + 120 \gamma + 155 \gamma^2 + 78 \gamma^3 + 13 \gamma^4) + 
      3 \beta^2 (24 + 120 \gamma + 206 \gamma^2 + 155 \gamma^3 + 50 \gamma^4 + 5 \gamma^5) + 
      \beta (28 + 168 \gamma + 360 \gamma^2 + 360 \gamma^3 + 171 \gamma^4 + 30 \gamma^5)) + 
   12 \alpha^2 (3 + 24 \gamma + 69 \gamma^2 + 96 \gamma^3 + 68 \gamma^4 + 20 \gamma^5 + 
      20 \beta^5 (1 + \gamma)^3 + 
      \beta^4 (68 + 272 \gamma + 366 \gamma^2 + 200 \gamma^3 + 36 \gamma^4) + 
      4 \beta^3 (24 + 120 \gamma + 206 \gamma^2 + 155 \gamma^3 + 50 \gamma^4 + 5 \gamma^5) + 
      2 \beta (12 + 84 \gamma + 207 \gamma^2 + 240 \gamma^3 + 136 \gamma^4 + 30 \gamma^5) + 
      \beta^2 (69 + 414 \gamma + 864 \gamma^2 + 824 \gamma^3 + 366 \gamma^4 + 60 \gamma^5)) + 
   2 \alpha (5 + 45 \gamma + 144 \gamma^2 + 224 \gamma^3 + 180 \gamma^4 + 60 \gamma^5 + 
      60 \beta^5 (1 + \gamma)^4 + 
      12 \beta^4 (15 + 75 \gamma + 136 \gamma^2 + 114 \gamma^3 + 43 \gamma^4 + 5 \gamma^5) + 
      12 \beta^2 (12 + 84 \gamma + 207 \gamma^2 + 240 \gamma^3 + 136 \gamma^4 + 30 \gamma^5) + 
      8 \beta^3 (28 + 168 \gamma + 360 \gamma^2 + 360 \gamma^3 + 171 \gamma^4 + 30 \gamma^5) + 
      3 \beta (15 + 120 \gamma + 336 \gamma^2 + 448 \gamma^3 + 300 \gamma^4 + 80 \gamma^5))\Big]^3$}
      
      \bigskip 
      
   \noindent 
      Remarkably, both the numerator and the denominator of this expression consist exclusively of  terms with positive coefficients. 
      In particular,  the second derivative of ${\mathcal A}$ along any  line segment
 $\alpha + \beta = \mbox{\bf const}$,  $\gamma = \widehat{\mbox{\bf const}}$, $\alpha, \beta > 0$,      is consequently everywhere positive, as claimed.    \end{proof}

\begin{lem}\label{symmetry3a}
In the  domain ${\zap U}\subset \check{\zap K}$ of our coordinates 
$(\alpha , \beta , \gamma ) \in (\RR^+)^3$, any critical point of ${\mathcal A}$
must lie on the plane $\alpha= \beta$. 
\end{lem}
\begin{proof} Notice
that ${\mathcal A}(\alpha, \beta , \gamma )= {\mathcal A}( \beta , \alpha,\gamma )$,
since there is an automorphism of $\CP_2\# 3 \overline{\CP}_2$ which 
exchanges the first two exceptional divisors.  
If $(\alpha, \beta , \gamma )$ were a critical point with $\alpha\neq \beta$,
we would have a second one given by $( \beta , \alpha,\gamma )$, 
and these two critical points would be joined by a line segment of the 
form  $\alpha + \beta = \mbox{\bf const}$,  
$\gamma = \widehat{\mbox{\bf const}}$, $\alpha, \beta > 0$.
However, this is a contradiction, because 
 Lemma \ref{convex3} tells us that  $\mathcal A$ is strictly convex on this
 line segment, so that it cannot possibly contain  two distinct
 critical points of $\mathcal A$. It follows that we  must have $\alpha=\beta$ for
 any critical point. \end{proof}

\begin{lem}\label{symmetry3b}
In the  domain ${\zap U}\subset \check{\zap K}$ of our coordinates 
$(\alpha , \beta , \gamma ) \in (\RR^+)^3$, any critical point of ${\mathcal A}$
must lie on the line  $\alpha= \beta=\gamma$. 
\end{lem}
\begin{proof} Notice
that ${\mathcal A}(\alpha, \beta , \gamma )= {\mathcal A}( \alpha,\gamma , \beta  )$,
since there is an automorphism of $\CP_2\# 3 \overline{\CP}_2$ which 
exchanges the last two exceptional divisors.  Conjugating by this
automorphism, 
 Lemma \ref{symmetry3a}
thus also shows that any critical point also satisfies $\alpha=\gamma$. Since any 
critical point also satisfies 
$\alpha=\beta$ by     Lemma \ref{symmetry3a},  
we therefore conclude that 
$\alpha= \beta=\gamma$ for any critical point in ${\zap U}$.
\end{proof}

 We now consider the action of two finite groups on $H^2(\CP_2\# 3 \overline{\CP}_2, \RR)$. 
 The first of these is the $\ZZ_2$-action induced by the Cremona transformation $\Phi$, 
 while the second  is the $\ZZ_3$-action generated by a cyclic permutation of  the 
 three blown-up points in $\CP_2$. 
 Let ${\mathbb V}$ 
(respectively,  ${\mathbb W}$) $\subset H^2(\CP_2\# 3 \overline{\CP}_2, \RR)$
 denote the invariant subspace (i.e. $(+1)$-eigenspace) of this 
 $\ZZ_2$-action (respectively, $\ZZ_3$-action).  
 In terms of the previously dicussed linear coordinates $(\alpha, \beta, \gamma, \delta )$ 
 on $H^2 ( \CP_2\# 3 \overline{\CP}_2, \RR)$, 
 ${\mathbb V}$ is  given by $\delta=0$, while 
 ${\mathbb W}$ is given by $\alpha=\beta=\gamma$. Notice
 that the anti-canonical class $c_1$ belongs to both   ${\mathbb V}$ and  ${\mathbb W}$. 
 Also notice that our two actions actually commute, so that 
  ${\mathbb V}$ and  ${\mathbb W}$ are in particular sent to themselves  $\Phi^*$. 
   
\begin{lem}\label{claritas}
If $\Omega \in {\zap K} \subset H^2(\CP_2\# 3 \overline{\CP}_2, \RR)$  is a 
critical point of ${\mathcal A}: {\zap K}\to \RR$, then either 
 $\Omega \in {\mathbb V}$ or  $\Omega \in {\mathbb W}$.
\end{lem}
\begin{proof} Recall that the reduced K\"ahler cone $\check{\zap K}={\zap K}/\RR^+$ has been 
divided into the coordinate domain ${\zap U}$, its image ${\zap U}^\prime$
under the Cremona transformation $\Phi$, and the interface ${\zap P}$.
Since ${\zap P}$ is exactly cut out by the equation $\delta=0$, it is exactly
the image of ${\zap K}\cap {\mathbb V}$ in $\check{\zap K}$. 
If $\Omega$ is critical, but does not belong to ${\mathbb V}$, then
either $\Omega$ or  $\Phi^*\Omega$ must project to  ${\zap U}$,
and in this case Lemma \ref{symmetry3b} tells us that 
either $\Omega$ or $\Phi^*\Omega$ must satisfy
$\alpha=\beta=\gamma$, and so belong to ${\mathbb W}$. 
Since  ${\mathbb W}$ is $\Phi$-invariant, the claim follows. 
\end{proof}

\begin{lem}\label{veritas}
Let $\Omega$ be a K\"ahler class that belongs to either ${\mathbb V}$ or
${\mathbb W}$. Then the Futaki invariant ${\mathfrak F}(\Omega)$
vanishes, and 
${\mathcal A} (\Omega )= (c_1\cdot \Omega )^2/\Omega^2$. 
\end{lem}
\begin{proof} By hypothesis, $\Omega$ is invariant under either our
$\ZZ_2$-action or our $\ZZ_3$-action. However, both of these
actions induce an action on the Lie algebra of the automorphism 
torus for which $+1$ is not an eigenvalue. Consequently \cite{calabix2,ls2},
the Futaki invariant must vanish for $\Omega$. This means that there
is no Futaki contribution to  ${\mathcal A}$, 
resulting in the stated simplification. 
\end{proof}

  \begin{prop} \label{gaudete} 
  For $M= \CP_2 \# 3 \overline{\CP}_2$, the 
   function ${\mathcal A}: \check{\zap K} \to \RR$ has exactly one 
   critical point. Moreover, this critical point is exactly the image of  $c_1\in {\zap K}$. 
     \end{prop}
\begin{proof}
Let $\Omega\in {\zap K}$ be a critical class. Then, by Lemma \ref{claritas},
$\Omega$ must belong to  one of the linear spaces ${\mathbb V}$ or
${\mathbb W}$. Since $c_1\in {\mathbb V}\cap {\mathbb W}$, it therefore
follows that $\Omega + t c_1$ belongs to $ {\mathbb V}\cup {\mathbb W}$
for all $t$. Moreover, since ${\zap K}$ is open,  $\Omega + t c_1$
is also a K\"ahler class for all $t$ in some small interval $(-\varepsilon, \varepsilon )$. 
By Lemma \ref{veritas}, we thus have
$${\mathcal A} (\Omega +t c_1) = \frac{[c_1\cdot (\Omega +tc_1)]^2}{(\Omega +tc_1)^2}$$
and  the first variation of $\mathcal A$ is therefore  given by 
 \begin{eqnarray*}
\left. \frac{d{\mathcal A}}{dt} \right|_{t=0}
&=& 
\left. \frac{d}{dt}  \frac{(c_1\cdot \Omega + tc_1^2 )^2}{(\Omega + tc_1)^2} \right|_{t=0}\\
&=& 
 \frac{2 c_1\cdot \Omega}{(\Omega^2)^2} \left[\Omega^2 c_1^2- (c_1\cdot \Omega)^2\right]~.
\end{eqnarray*}
However, because  the intersection form on $H^2 (\CP_2 \# 3\overline{\CP}_2, \RR)$
 is of Lorentz type, and because $\Omega$ and $c_1$ are both time-like,  
 the reverse Cauchy-Schwarz 
 inequality for Lorentzian inner products tells us that 
 $$
 (c_1\cdot \Omega)^2 \leq  c_1^2 ~\Omega^2~,
 $$
 with equality iff $\Omega$ and  $c_1$ are linearly dependent;
 consequently,  the  first variation computed above is
automatically
negative unless $\Omega$ is a multiple of $c_1$. It follows that 
$\Omega$ can be a critical point of $\mathcal A$ only if its image in $\check{\zap K}$
coincides with that  of $c_1$. Moreover, by (\ref{sharp}) and the reverse
 Cauchy-Schwarz inequality, $c_1$ is actually  the unique absolute minimum 
 of ${\mathcal A}$, and so, in particular, is a critical point. 
This  establishes the stated uniqueness result. 
\end{proof}

  Theorem \ref{uni3} now follows by the same reasoning we previously used
  to prove Theorem \ref{uni2}. In particular, a conformally Einstein K\"ahler metric
   $M= \CP_2 \# 3 \overline{\CP}_2$ must actually be K\"ahler-Einstein, and so 
   does not give rise to an entry on the list of exceptions in Theorem \ref{jubilo}.
   
 \bigskip   
   
   Combining  Proposition \ref{aarhus} with Theorems \ref{laudate} and \ref{gaudete},
  we have thus proved  Theorem  \ref{jubilo}, and our chain of reasoning  has thus 
culminated in the proof of all the promised results.   

\pagebreak

\vfill 

\noindent 
{\bf Acknowledgement.} The author would like to warmly thank 
Gideon Maschler for  numerous discussions which helped 
lay the groundwork for the present investigation. 
  
\vspace{1in}

\noindent
{\sc Author's address:} 

\medskip 

 \noindent
{Mathematics Department, SUNY, Stony Brook, NY 11794, USA
}

\bigskip

\noindent
{\sc Author's e-mail:} 

\medskip 

 \noindent
{\tt claude@math.sunysb.edu
}


\begin{thebibliography}{10}

\bibitem{aggs}
{\sc V.~Apostolov and P.~Gauduchon}, {\em The {R}iemannian {G}oldberg-{S}achs
  theorem}, Internat. J. Math., 8 (1997), pp.~421--439.

\bibitem{aubin}
{\sc T.~Aubin}, {\em Equations du type {M}onge-{A}mp\`{e}re sur les
  vari\'{e}t\'{e}s k{\"a}hl{\'e}riennes compactes}, C. R. Acad. Sci. Paris,
  283A (1976), pp.~119--121.

\bibitem{beber3}
{\sc L.~B{\'e}rard-Bergery}, {\em Sur de nouvelles vari\'et\'es riemanniennes
  d'{E}instein}, in Institut \'{E}lie {C}artan, 6, vol.~6 of Inst. \'Elie
  Cartan, Univ. Nancy, Nancy, 1982, pp.~1--60.

\bibitem{bes}
{\sc A.~L. Besse}, {\em Einstein manifolds}, vol.~10 of Ergebnisse der
  Mathematik und ihrer Grenzgebiete (3), Springer-Verlag, Berlin, 1987.

\bibitem{calabix}
{\sc E.~Calabi}, {\em Extremal {K}\"ahler metrics}, in Seminar on Differential
  Geometry, vol.~102 of Ann. Math. Studies, Princeton Univ. Press, Princeton,
  N.J., 1982, pp.~259--290.

\bibitem{calabix2}
\leavevmode\vrule height 2pt depth -1.6pt width 23pt, {\em Extremal {K}\"ahler
  metrics. {II}}, in Differential {G}eometry and {C}omplex {A}nalysis,
  Springer, Berlin, 1985, pp.~95--114.

\bibitem{xxel}
{\sc X.~Chen}, {\em Space of {K}\"ahler metrics. {III}. {O}n the lower bound of
  the {C}alabi energy and geodesic distance}, Invent. Math., 175 (2009),
  pp.~453--503.

\bibitem{chenlebweb}
{\sc X.~X. Chen, C.~LeBrun, and B.~Weber}, {\em On conformally {K}\"ahler,
  {E}instein manifolds}, J. Amer. Math. Soc., 21 (2008), pp.~1137--1168.

\bibitem{xxgang2}
{\sc X.~X. Chen and G.~Tian}, {\em Geometry of {K}\"ahler metrics and
  foliations by holomorphic discs}, Publ. Math. Inst. Hautes \'Etudes Sci.,
  (2008), pp.~1--107.

\bibitem{delpezzo}
{\sc M.~Demazure}, {\em Surfaces de del {P}ezzo, {II}, {III}, {IV}, {V}}, in
  S\'eminaire sur les {S}ingularit\'es des {S}urfaces, vol.~777 of Lecture
  Notes in Mathematics, Berlin, 1980, Springer, pp.~21--69.

\bibitem{derd}
{\sc A.~Derdzi{\'n}ski}, {\em Self-dual {K}\"ahler manifolds and {E}instein
  manifolds of dimension four}, Compositio Math., 49 (1983), pp.~405--433.

\bibitem{donaldsonk1}
{\sc S.~K. Donaldson}, {\em Scalar curvature and projective embeddings. {I}},
  J. Differential Geom., 59 (2001), pp.~479--522.

\bibitem{fuma0}
{\sc A.~Futaki and T.~Mabuchi}, {\em Bilinear forms and extremal {K}\"ahler
  vector fields associated with {K}\"ahler classes}, Math. Ann., 301 (1995),
  pp.~199--210.

\bibitem{gs}
{\sc J.~N. Goldberg and R.~K. Sachs}, {\em A theorem on {P}etrov types}, Acta
  Phys. Polon., 22 (1962), pp.~13--23.

\bibitem{nur}
{\sc A.~R. Gover, C.~D. Hill, and P.~Nurowski}, {\em Sharp version of the
  {G}oldberg-{S}achs theorem}.
\newblock e-print arXiv:0911.3364 [math.DG], 2009.

\bibitem{hit}
{\sc N.~J. Hitchin}, {\em On compact four-dimensional {E}instein manifolds}, J.
  Differential Geom., 9 (1974), pp.~435--442.

\bibitem{lmo}
{\sc C.~LeBrun}, {\em {E}instein metrics and {M}ostow rigidity}, Math. Res.
  Lett., 2 (1995), pp.~1--8.

\bibitem{lebhem}
\leavevmode\vrule height 2pt depth -1.6pt width 23pt, {\em Einstein metrics on
  complex surfaces}, in Geometry and {P}hysics ({A}arhus, 1995), vol.~184 of
  Lecture Notes in Pure and Appl. Math., Dekker, New York, 1997, pp.~167--176.

\bibitem{lebhem10}
{\sc C.~LeBrun}, {\em Einstein manifolds and extremal K\"ahler metrics}.
\newblock e-print arXiv:1009.1270 [math.DG], 2010.

\bibitem{ls2}
{\sc C.~LeBrun and S.~R. Simanca}, {\em On the {K}\"ahler classes of extremal
  metrics}, in Geometry and Global Analysis (Sendai, 1993), Tohoku Univ.,
  Sendai, 1993, pp.~255--271.

\bibitem{ls}
\leavevmode\vrule height 2pt depth -1.6pt width 23pt, {\em Extremal {K}\"ahler
  metrics and complex deformation theory}, Geom. Funct. Anal., 4 (1994),
  pp.~298--336.

\bibitem{mabuniq}
{\sc T.~Mabuchi}, {\em Uniqueness of extremal {K}\"ahler metrics for an
  integral {K}\"ahler class}, Internat. J. Math., 15 (2004), pp.~531--546.

\bibitem{cubic}
{\sc Y.~I. Manin}, {\em Cubic {F}orms: {A}lgebra, {G}eometry, {A}rithmetic},
  North-Holland Publishing Co., Amsterdam, 1974.
\newblock Translated from the Russian by M. Hazewinkel.

\bibitem{gideon2}
{\sc G.~Maschler}, {\em Uniqueness of {E}instein metrics conformal to extremal
  {K}\"ahler metrics---a computer assisted approach}, AIP Conf. Proc., 1093
  (2009), pp.~132--143.

\bibitem{page}
{\sc D.~Page}, {\em A compact rotating gravitational instanton}, Phys. Lett.,
  79B (1979), pp.~235--238.

\bibitem{pb}
{\sc M.~Przanowski and B.~Broda}, {\em Locally {K}\"ahler gravitational
  instantons}, Acta Phys. Polon. B, 14 (1983), pp.~637--661.

\bibitem{s}
{\sc Y.~Siu}, {\em The existence of {K\"a}hler-{E}instein metrics on manifolds
  with positive anti-canonical line bundle and suitable finite symmetry group},
  Ann. Math., 127 (1988), pp.~585--627.

\bibitem{tian}
{\sc G.~Tian}, {\em On {C}alabi's conjecture for complex surfaces with positive
  first {C}hern class}, Inv. Math., 101 (1990), pp.~101--172.

\bibitem{ty}
{\sc G.~Tian and S.~T. Yau}, {\em {K\"a}hler-{E}instein metrics on complex
  surfaces with ${\bf c}_1 > 0$}, Comm. Math. Phys., 112 (1987), pp.~175--203.

\bibitem{yau}
{\sc S.~T. Yau}, {\em {C}alabi's conjecture and some new results in algebraic
  geometry}, Proc. Nat. Acad. USA, 74 (1977), pp.~1789--1799.

\end{thebibliography}
  \end{document}